\documentclass[a4paper]{amsart}
\usepackage{wrapfig}
\usepackage[dvips]{graphicx}
\usepackage{amsmath,amsthm,amssymb,amscd}
\usepackage{mathrsfs}
\theoremstyle{definition}
\newtheorem{thm}{Theorem}[section]
\newtheorem{Def}[thm]{Definition}
\newtheorem{pro}[thm]{Proposition}
\newtheorem{cor}[thm]{Corollary}
\newtheorem{lem}[thm]{Lemma}

\newtheorem{rem}[thm]{Remark}

\newtheorem*{mainthm}{Theorem}
\theoremstyle{definition}

\begin{document}
\title{Rohlin actions of finite groups on the Razak-Jacelon algebra}
\author{Norio Nawata}
\address{Department of Educational Collaboration, Osaka Kyoiku University, 4-698-1 Asahigaoka, Kashiwara, Osaka, 582-8582, Japan}
\email{nawata@cc.osaka-kyoiku.ac.jp}
\thanks{This work was supported by JSPS KAKENHI Grant Number 16K17614}

\begin{abstract}
Let $A$ be a simple separable nuclear C$^*$-algebra with a unique tracial state and no unbounded 
traces, and let $\alpha$ be a strongly outer action of a finite group $G$ on $A$. In this paper, we show that $\alpha\otimes \mathrm{id}$ on $A\otimes\mathcal{W}$ has the Rohlin property, where 
$\mathcal{W}$ is the Razak-Jacelon algebra. 
Combing  this result with the recent classification results and our previous result, we see that such 
actions are unique up to conjugacy. 
\end{abstract}
\maketitle

\section{Introduction} 

Let $\mathcal{O}_2$ be the Cuntz algebra generated by 2 isometries. 
It is known that $\mathcal{O}_2$ is a simple separable unital nuclear purely infinite  
C$^*$-algebra, and is $KK$-equivalent to $\{0\}$. 
Kirchberg and Phillips showed that a simple separable unital nuclear C$^*$-algebra $B$ is 
isomorphic to $\mathcal{O}_2$ if and only if  $B$ has an asymptotically central inclusion of 
$\mathcal{O}_2$ in \cite{KP}. In particular, if $A$ is a simple separable unital nuclear C$^*$-algebra, 
then $A\otimes\mathcal{O}_2$ is isomorphic to $\mathcal{O}_2$. 
It is known that $\mathcal{O}_2$ plays an important role in the classification of nuclear 
C$^*$-algebra (see, for example, \cite{G2} and \cite{Ror1}). 

Let $\mathcal{W}$ be the Razak-Jacelon algebra studied in \cite{J}, which is a  certain simple 
separable nuclear stably projectionless C$^*$-algebra having trivial $K$-groups and a unique 
tracial state and no unbounded traces. Note that $\mathcal{W}$ is $KK$-equivalent to $\{0\}$ 
and $\mathcal{O}_2$. 
Hence we may regard $\mathcal{W}$ as a stably finite analogue of $\mathcal{O}_2$. 
Combing Elliott, Gong, Lin and Niu's result \cite{EGLN} and  Castillejos and Evington's result 
\cite{CE} (see also \cite{CETWW}), we see that if $A$ is a simple separable nuclear 
C$^*$-algebra with a unique tracial state and no unbounded traces, then $A\otimes\mathcal{W}$ 
is isomorphic to $\mathcal{W}$. 
We refer the reader to \cite{EGLN0}, \cite{EGLN} (see also \cite{EN} and \cite{GL}) and \cite{GL2} 
for recent progress in the classification of stably projectionless C$^*$-algebras. 

In the theory of operator algebras, the classification of group actions is one of the most 
fundamental problems and has a long history. There exists a complete classification 
of actions of countable amenable groups on approximately finite dimensional (AFD) factors. 
Although there exist some successes in the classification of  group actions on ``classifiable'' 
C$^*$-algebras, the classification of countable amenable group (outer) actions on ``classifiable'' 
C$^*$-algebras  is far from complete because of $K$-theoretical obstructions. 
We refer the reader to \cite{I} and the references given there for details and results in 
the classification of group actions on operator algebras. 
We shall review only some results that are directly related to this paper. 

Connes \cite{C3} classified finite cyclic group actions on the AFD factor 
$\mathcal{R}_0$ of type II$_1$ up to conjugacy. More generally, Jones \cite{Jones} classified 
finite group actions on $\mathcal{R}_0$. 
In particular, outer actions of a finite group on $\mathcal{R}_0$ are unique up to conjugacy. 

In \cite{I1}, Izumi introduced the Rohlin property of finite group actions on unital C$^*$-algebras and 
showed an equivariant version of the Kirchberg-Phillips type theorem for finite group actions 
on $\mathcal{O}_2$. 
Indeed, he characterized Rohlin actions on $\mathcal{O}_2$ by using the fixed point subalgebra 
of the central sequence C$^*$-algebra of $\mathcal{O}_2$ and showed that if $\alpha$ is an 
outer action of a finite group $G$ on a simple separable unital nuclear C$^*$-algebra $A$, then 
$\alpha\otimes\mathrm{id}$ on $A\otimes\mathcal{O}_2$ has the Rohlin property. 
In particular, such actions are unique up to conjugacy. Note that Izumi also showed that there exist 
uncountably many mutually non-conjugate outer actions of $\mathbb{Z}_2$ on $\mathcal{O}_2$. 
Also, Goldstein and Izumi obtained an equivariant Kirchberg-Phillips type result for finite group 
actions on $\mathcal{O}_{\infty}$ in \cite{GI}. 

Remarkably, Szab\'o generalized Izumi's result to countable amenable group actions in \cite{Sza4}. 
He showed that countable amenable group outer actions on $\mathcal{O}_2$ that equivariantly 
absorb the trivial action on $\mathcal{O}_2$ are unique up to strong cocycle conjugacy. 
Note that  Szab\'o considered more general settings and obtained results for  strongly 
self-absorbing C$^*$-dynamical systems. See \cite{Sza3}, \cite{Sza1}, \cite{Sza2}, \cite{Sza4}, 
\cite{Sza5} and \cite{Sza6}. 

In this paper, we shall consider an equivariant Kirchberg-Phillips type result for finite group actions on 
$\mathcal{W}$. Indeed, we shall show that if $\alpha$ is a strongly outer action of a finite group 
$G$ on a simple separable nuclear C$^*$-algebra $A$ with a unique tracial state and no unbounded 
traces, then $\alpha\otimes\mathrm{id}$ on $A\otimes\mathcal{W}$ has the Rohlin property 
(Theorem \ref{thm:main}). 
Since the author showed that Rohlin actions of a finite group on $\mathcal{W}$ are unique up 
to conjugacy in \cite{Na4}, we see that such actions are unique up to conjugacy by Elliott, Gong, 
Lin and Niu's result and Castillejos and Evington's result. Indeed, we obtain the following theorem.

\begin{mainthm} (Corollary \ref{cor:main}) \ \\
Let $A$ and $B$ be simple separable nuclear C$^*$-algebras with a unique tracial state and 
no unbounded traces, and let $\alpha$ and $\beta$ be strongly outer actions of a finite group 
$G$ on $A$ and $B$, respectively. 
Then $\alpha\otimes\mathrm{id}$ on $A\otimes\mathcal{W}$ is conjugate to 
$\beta\otimes\mathrm{id}$ on $B\otimes\mathcal{W}$. 
\end{mainthm}

Our main result (Theorem \ref{thm:main}) is shown by using a cohomolgy vanishing type  result 
(Lemma \ref{lem:cohomology}). 
The proof of Lemma \ref{lem:cohomology} is based on Connes' $2\times 2$ matrix trick in 
\cite[Corollary 2.6]{C3}. 
We need to consider the comparison theory for projections in the fixed point subalgebra 
$F(A\otimes\mathcal{W})^{\alpha\otimes\mathrm{id}}$ of 
the central sequence C$^*$-algebra of $A\otimes\mathcal{W}$ for Connes' $2\times 2$ matrix trick. 
We obtain this as a corollary of a classification up to unitary equivalence of certain normal 
elements in $F(A\otimes\mathcal{W})^{\alpha\otimes\mathrm{id}}$. This classification is 
based on arguments in \cite{Na3} where the author classified certain unitary elements and 
projections in $F(\mathcal{W})$ up to unitary equivalence. 

This paper is organized as follows. In Section \ref{sec:Pre}, we collect notations, definitions and 
some results. In Section \ref{sec:target} and Section \ref{sec:stable-uniqueness}, we show 
a variant of \cite[Corollary 3.8]{Na3}, which is a main technical tool in this paper.  
In particular, we introduce a (non-separable) C$^*$-algebra $\mathcal{B}^{\gamma}$, and show 
that $\mathcal{B}^{\gamma}$ has strict comparison (Proposition \ref{pro:main-section3}) 
in Section \ref{sec:target}. 
Note that $\mathcal{B}^{\gamma}$ is a target algebra of a (natural) homomorphism from 
$F(A\otimes\mathcal{W})^{\alpha\otimes\mathrm{id}}\otimes\mathcal{W}$. 
The proof of Proposition \ref{pro:main-section3} is essentially based 
on arguments in \cite{MS}, \cite{MS2} and \cite{MS3}. In particular, it is important to consider 
the property (SI) and the weak Rohlin property. 
These concepts were introduced by Sato in his pioneering work \cite{Sa0} 
and \cite{Sa}  (see also \cite{Kis1}). We refer the reader to \cite{Sa1} for 
recent progress of such type arguments. 
Section \ref{sec:stable-uniqueness} is essentially based on arguments in 
\cite{EN} (see \cite[Section 3]{Na3}). 
In Section \ref{sec:normal}, we classify certain normal elements in 
$F(A\otimes\mathcal{W})^{\alpha\otimes\mathrm{id}}$ up to unitary equivalence (Theorem 
\ref{thm:classification-normal}), and show a comparison theorem for certain projections in 
$F(A\otimes\mathcal{W})^{\alpha\otimes\mathrm{id}}$ (Corollary \ref{cor:comparison}). 
In Section \ref{sec:main}, we show the main result in this paper. 

\section{Preliminaries}\label{sec:Pre}

In this section we shall collect notations, definitions and some results. 

For a C$^*$-algebra $A$, let $A_{+}$ denote the set of positive elements in $A$ and $A_{+,1}$ 
the set of positive contractions in $A$. For $x,y\in A$, let $[x,y]$ be the commutator $xy-yx$. 
We denote by $K(H)$ and $M_{n^\infty}$ for $n\in\mathbb{N}$ 
the C$^*$-algebra of compact operators on a Hilbert space $H$ 
and the uniformly hyperfinite (UHF) algebra of type $n^{\infty}$, respectively.

\subsection{Approximate units and actions} 
 If $A$ is a separable C$^*$-algebra, then there exists a positive element $s\in A$ such that 
$sA$ is dense in $A$. Such a positive element $s$ is said to be {\it strictly positive} in $A$. 
For any $n\in\mathbb{N}$, define $f_n:[0,1]\to \mathbb{R}$ by 
$$
f_n(t):=\left\{\begin{array}{cl}
0 & t\in [0,\frac{1}{n+1}]  \\
n(n+1)t-n & t\in (\frac{1}{n+1},\frac{1}{n}] \\
1 & t\in (\frac{1}{n}, 1]
\end{array}
\right..
$$
If $s$ is a strictly positive element in $A$ and $\|s\|=1$, then $\{f_n(s)\}_{n\in\mathbb{N}}$ is 
an approximate unit for $A$ with $f_{n+1}(s)f_{n}(s)=f_{n}(s)$.  
Let $A^{\sim}$ denote the unitization algebra of $A$. Note that we assume $A^{\sim}=A$ if 
$A$ is unital. 
Let $M(A)$ be the \textit{multiplier algebra} of $A$, which is the largest unital C$^*$-algebra 
that contains $A$ as an essential ideal. 
If $\alpha$ is an automorphism of $A$, then $\alpha$ extends uniquely to an automorphism of 
$M(A)$. We denote it by the same symbol $\alpha$ for simplicity. 

We denote by $\mathrm{Aut}(A)$ the automorphism group of $A$. 
An automorphism $\alpha$ of $A$ is said to be \textit{inner} if there exists a unitary element $u$ in 
$M(A)$ such that $\alpha (x)=\mathrm{Ad}(u)(x)=uxu^*$ for any $x\in A$. 
For a subset $F$ of $A$ and $\varepsilon >0$, we say a completely positive (c.p.) map 
$\varphi :A\to B$ is \textit{$(F,\varepsilon)$-multiplicative} if 
$$
\| \varphi (xy) - \varphi (x)\varphi(y) \| < \varepsilon 
$$
for any $x,y\in F$. 

An \textit{action} $\alpha$ of a discrete group $G$ on $A$ is a homomorphism from $G$ to 
$\mathrm{Aut}(A)$. We say that $\alpha$ is \textit{outer} if $\alpha_g$ is not inner 
for any $g\in G\setminus \{\iota \}$ where $\iota$ is the identity of $G$. 
An $\alpha$-\textit{cocycle} is a map $w$ from $G$ to the unitary group of $M(A)$ such that 
$w(gh)=w(g)\alpha_{g}(w(h))$ for any $g,h\in G$. 
We say that an $\alpha$-cocycle $w$ is a \textit{coboundary} if there exists a unitary element $v$ 
in $M(A)$ such that $w(g)=v\alpha_g(v^*)$ for any $g\in G$. 
For two $G$-actions $\alpha$ on $A$ and $\beta$ on $B$, we say that $\alpha$ and $\beta$ are 
\textit{conjugate} if there exists an isomorphism $\theta$ from $A$ onto $B$ such that 
$\theta \circ \alpha_g=\beta_{g}\circ \theta$ for any $g\in G$. 
We denote by $A^{\alpha}$ the fixed point algebra. 

Every tracial state $\tau$ on $A$ extends uniquely to a tracial state on $M(A)$.  
We denote it by the same symbol $\tau$ for simplicity. 
Let $(\pi_{\tau}, H_{\tau})$ be the Gelfand-Naimark-Segal (GNS) representation of $A$ associated 
with $\tau$. 
Then $\tau$ extends uniquely to a normal tracial state $\tilde{\tau}$ on $\pi_{\tau} (A)^{''}$. 
If $\alpha$ is an automorphism of $A$ such that $\tau \circ \alpha =\tau$, 
then $\alpha$ extends uniquely to an automorphism $\tilde{\alpha}$ of $\pi_{\tau} (A)^{''}$. 
Moreover if $\alpha$ is an action of $G$ on $A$ such that  $\tau \circ \alpha_g =\tau$ for any 
$g\in G$, then 
$\alpha$ extends uniquely to a von Neumann algebraic action $\tilde{\alpha}$ on $\pi_{\tau}(A)^{''}$. 
We say that an action $\alpha$ of $G$ on a C$^*$-algebra $A$ with a unique tracial state $\tau$ 
is \textit{strongly outer} if $\tilde{\alpha}_g$ is not inner in $\pi_{\tau}(A)^{''}$ for any 
$g\in G\setminus \{\iota\}$.

\subsection{Kirchberg's central sequence C$^*$-algebras}
We shall recall Kirchberg's central sequence C$^*$-algebras in \cite{Kir2} 
(see also \cite[Section 5]{Na2} and \cite[Section 2.2]{Na3}). 
Fix a free ultrafilter $\omega$ on $\mathbb{N}$. For a C$^*$-algebra $A$, put 
$$
c_{\omega}(A):=\{\{x_n\}_{n\in\mathbb{N}}\in \ell^{\infty}(\mathbb{N}, A)\; 
|\; \lim_{n \to \omega}\| x_n\| =0 \}, \; 
A^{\omega}:=\ell^{\infty}(\mathbb{N}, A)/c_{\omega}(A). 
$$
A sequence $(x_n)_n$ is a representative of an element in $A^{\omega}$. 
Let $B$ be a C$^*$-subalgebra of $A$. 
We identify $A$ and $B$ with the C$^*$-subalgebras of $A^\omega$ consisting of equivalence 
classes of constant sequences.  Set 
$$
A_{\omega}:=A^{\omega}\cap A^{\prime},\; \mathrm{Ann}(B,A^{\omega}):=\{(x_n)_n\in A^{\omega}\cap B^{\prime}\; |\; (x_n)_nb =0
\;\mathrm{for}\;\mathrm{any}\; b\in B \}.
$$
Then $\mathrm{Ann}(B,A^{\omega})$ is a closed ideal of $A^{\omega}\cap B^{\prime}$. 
Define a \textit{central sequence C$^*$-algebra} $F(A)$ of $A$ by 
$$
F(A):=A_{\omega}/\mathrm{Ann}(A,A^{\omega}).
$$ 
If $\{h_n\}_{n\in\mathbb{N}}$ is a countable approximate unit for $A$, then $[(h_n)_n]$ is a unit in $F(A)$.  
It can be easily checked that $F(A)$ is isomorphic to $M(A)^\omega\cap A^{\prime}/\mathrm{Ann}(A,M(A)^{\omega})$ and ${A^{\sim}}_{\omega}/ \mathrm{Ann}(A,(A^{\sim})^{\omega})$. 
If $\alpha$ is an automorphism of $A$, $\alpha$ induces natural automorphisms of $A^{\omega}$, $A_{\omega}$ and $F(A)$. 
We denote them by the same symbol $\alpha$ for simplicity. 
For a tracial state $\tau$ on $A$, define $\tau_{\omega}([(x_n)_n]):=\lim_{n\to\omega}\tau (x_n)$. 
Then $\tau_{\omega}$ is a well defined tracial state on $F(A)$ by \cite[Proposition 2.1]{Na3}.

\subsection{Razak-Jacelon algebra}

Let $\mathcal{W}$ be the Razak-Jacelon algebra studied in \cite{J}, 
which is a simple separable nuclear C$^*$-algebra with a unique tracial state and no unbounded 
traces, and is $KK$-equivalent to $\{0\}$. 
The Razak-Jacelon algebra $\mathcal{W}$ is constructed as an inductive limit 
C$^*$-algebra of Razak's building block in \cite{Raz}. 
Let $S_1$ and $S_2$ be the generators of  the Cuntz algebra $\mathcal{O}_2$. 
For every $\lambda_1,\lambda_2\in\mathbb{R}$, define a flow 
$\gamma$ on $\mathcal{O}_2$ by $\gamma_t (S_j)=e^{it\lambda_{j}}S_j$. 
Kishimoto and Kumjian showed that if 
$\lambda_{1}$ and $\lambda_{2}$ are all non-zero, of the same sign and 
$\lambda_1$ and $\lambda_2$ generate $\mathbb{R}$ as a closed subgroup, then 
$\mathcal{O}_2\rtimes_{\gamma}\mathbb{R}$ is a simple stably projectionless C$^*$-algebra 
with unique (up to scalar multiple) trace in \cite{KK1} and \cite{KK2}. 
Robert showed that $\mathcal{W}\otimes K(\ell^2(\mathbb{N}))$ is isomorphic to 
$\mathcal{O}_2\rtimes_{\gamma}\mathbb{R}$ for some $\lambda_1$ and $\lambda_2$ in \cite{Rob}. 
(See also \cite{Dean}.) 
Razak's classification theorem \cite{Raz} implies that $\mathcal{W}$ is UHF-stable, and hence 
$\mathcal{W}$ is $\mathcal{Z}$-stable. 

\subsection{Corollaries of Matui and Sato's results}

We shall collect some corollaries of  Matui and Sato's results in \cite{MS1} and \cite{MS2}. 
Although they assume that C$^*$-algebras are unital, 
their arguments for the following results work for non-unital C$^*$-algebras by suitable modifications 
(see \cite{Na2} and \cite{Na3}). 

First, we recall the definition of the weak Rohlin property. See \cite[Definition 2.7]{MS1} and  
\cite[Definition 2.5]{MS2}. Note that Matui and Sato define the weak Rohlin property for more 
general settings. 

\begin{Def} 
Let $A$ be a simple C$^*$-algebra with a unique tracial state $\tau$, and let $\alpha$ be an action of 
a finite group $G$ on $A$. We say that $\alpha$ has the weak Rohlin property if there exists a 
positive contraction $f$ in $F(A)$ such that 
$$
\alpha_g(f)\alpha_h(f)=0, \quad \tau_{\omega}(f)= \frac{1}{|G|} 
$$
for any $g,h\in G$ with $g\neq h$. 
\end{Def}

Essentially the same proof as \cite[Theorem 3.4]{MS1} shows the following theorem. 
See also the proof of \cite[Lemma 6.2]{Na3} and \cite[Theorem 3.6]{MS2}.  

\begin{thm}\label{thm}\label{thm:weak-rohlin}
Let $A$ be a simple separable nuclear C$^*$-algebra with a unique tracial state and no unbounded 
traces, and let $\alpha$ be an action of a finite group $G$ on $A$. Then $\alpha$ has the weak 
Rohlin property if and only if $\alpha$ is strongly outer. 
\end{thm}

Essentially the same proofs as \cite[Lemma 4.7]{MS2} and \cite[Proposition 4.8]{MS2} show 
the following proposition. 
See also \cite[Propostion 3.3]{MS3} and \cite[Theorem 4.1]{BBSTWW}. 
Note that if $A$ is a simple separable nuclear C$^*$-algebra with 
a unique tracial state and no unbounded traces, then $A\otimes\mathcal{W}$ has property (SI) 
since $\mathcal{W}$ is $\mathcal{Z}$-stable (see \cite{Ror}, \cite{MS} and \cite{Na2}).

\begin{pro}\label{thm:Matui-Sato} 
Let $A$ be a simple separable nuclear C$^*$-algebra with a unique tracial state 
and no unbounded traces, and let $\alpha$ be a strongly outer action of a finite group 
$G$ on $A$. Then: \ \\
(i) $F(A\otimes\mathcal{W})^{\alpha\otimes\mathrm{id}}$ has a unique tracial state $\tau_{\omega}$. 
\ \\
(ii) If $a$ and $b$ are positive elements in $F(A\otimes\mathcal{W})^{\alpha\otimes\mathrm{id}}$
satisfying $d_{\tau_{\omega}}(a)< d_{\tau_{\omega}} (b)$, then there exists an element 
$r\in F(A\otimes\mathcal{W})^{\alpha\otimes\mathrm{id}}$ such that $r^*br=a$. 
\end{pro}

\subsection{Rohlin property and properties of $F(A\otimes\mathcal{W})^{\alpha\otimes\mathrm{id}}$}

We shall recall some results in \cite{Na4} (see also \cite{GS}) and \cite{Na3}. 

\begin{Def} (cf. \cite[Definition 3.1]{I1} and \cite[Definition 3.1]{Na4}). 
An action $\alpha$ of a finite group $G$ on a separable C$^*$-algebra $A$ is said to have the 
\textit{Rohlin property} if there exists a partition of unity $\{p_{g}\}_{g\in G}\subset F(A)$ consisting 
of projections satisfying 
$$
\alpha_{g} (p_{h}) =p_{gh},
$$
for any $g,h\in G$.
\end{Def}

For any finite group $G$, there exists an action of $G$ on $\mathcal{W}$ with the Rohlin property 
by \cite[Example 3.2]{Na4}. The following theorem is \cite[Corollary 3.7]{Na4}.

\begin{thm}\label{thm:classification}
Let $\alpha$ and $\beta$ be actions of a finite group $G$ on $\mathcal{W}$ with the Rohlin property. 
Then $\alpha$ and $\beta$ are conjugate. 
\end{thm}

Note that there exists a strongly outer action $\alpha$ of $\mathbb{Z}_2$ on $\mathcal{W}$ 
such that $\alpha$ does not have the Rohlin property (see \cite[Example 5.6]{Na4}). 

Since we can regard $F(\mathcal{W})$ is a unital C$^*$-subalgebra of 
$F(A\otimes\mathcal{W})^{\alpha\otimes\mathrm{id}}$, 
we obtain the following proposition by \cite[Proposition 4.2]{Na3} and Proposition \ref{thm:Matui-Sato}. 

\begin{pro}\label{pro:key-pro} Let $\tau_{\omega}$ be the unique tracial state on 
$F(A\otimes\mathcal{W})^{\alpha\otimes\mathrm{id}}$. \ \\
(i) For any $N\in\mathbb{N}$, there exists a unital homomorphism from $M_N(\mathbb{C})$ to 
$F(A\otimes\mathcal{W})^{\alpha\otimes\mathrm{id}}$. \ \\
(ii) For any $\theta\in [0,1]$, there exists a non-zero projection $p$ in 
$F(A\otimes\mathcal{W})^{\alpha\otimes\mathrm{id}}$ such that $\tau_{\omega}(p)=\theta$. \ \\
(iii) Let $h$ be a positive element in $F(A\otimes\mathcal{W})^{\alpha\otimes\mathrm{id}}$ 
such that $d_{\tau_{\omega}}(h)>0$. For any $\theta \in [0, d_{\tau_{\omega}}(h))$, 
there exists a non-zero projection $p$ in $\overline{hF(A\otimes\mathcal{W})^{\alpha\otimes
\mathrm{id}}h}$ such that $\tau_{\omega}(p)=\theta$. 
\end{pro}

Using the proposition above instead of \cite[Proposition 4.2]{Na3}, the same arguments as in 
\cite[Section 4]{Na3} show the following proposition. 

\begin{pro}\label{pro:MvN-u} (cf. \cite[Proposition 4.8]{Na3}). 
Let $p$ and $q$ be projections in $F(A\otimes\mathcal{W})^{\alpha\otimes\mathrm{id}}$ 
such that $\tau_{\omega} (p)<1$ where $\tau_{\omega}$ is the unique tracial state on 
$F(A\otimes\mathcal{W})^{\alpha\otimes\mathrm{id}}$. Then $p$ and $q$ are Murray-von Neumann 
equivalent if and only if $p$ and $q$ are unitarily equivalent. 
\end{pro}

\section{Target algebra}\label{sec:target}

In the rest of this paper, we assume that $A$ is a simple separable nuclear 
C$^*$-algebra with a unique tracial state $\tau_A$ and no unbounded traces, and $\alpha$ is 
a strongly outer action of a finite group $G$ on $A$. 
Define an action $\gamma$ on $A\otimes\mathcal{W}$ by $\gamma:= \alpha\otimes\mathrm{id}$. 
Let $\tau_{\mathcal{W}}$ denote the unique tracial state on $\mathcal{W}$, 
and let $\tau:=\tau_A\otimes\tau_{\mathcal{W}}$ 
on $A\otimes\mathcal{W}$. 
For any $a\in A$ and $b\in \mathcal{W}$, we regard $a\otimes 1_{\mathcal{W}^{\sim}}$ and 
$1_{A^{\sim}}\otimes b$ as elements in $M(A\otimes\mathcal{W})$. Put 
$$
\mathcal{A}:= \{(x_n)_n\in (A\otimes\mathcal{W})^{\omega}\; |\; 
([x_n ,a\otimes 1_{\mathcal{W}^{\sim}}])_n=0 \text{ for any }a\in A\}
$$
and 
$$
\mathcal{I}:= \{(x_n)_n\in\mathcal{A} \; |\; (x_n (a\otimes 1_{\mathcal{W}^{\sim}}))_n=0 
\text{ for any }a\in A \}.
$$
Then $\mathcal{I}$ is a closed ideal of $\mathcal{A}$, and define  
$\mathcal{B}:= \mathcal{A}/\mathcal{I}$. 
Note that for any $[(x_n)_n]\in \mathcal{B}$,
$$
\| [(x_n)_n] \| = \sup_{a\in A_{+,1}} \lim_{n\to \omega} \|x_n(a\otimes 1_{\mathcal{W}^{\sim}}) \|.
$$
Indeed, let $\| [(x_n)_n] \|^{\prime}:=\sup_{a\in A_{+,1}}\lim_{n\to \omega} \|x_n
(a\otimes 1_{\mathcal{W}^{\sim}}) \|$ for any $[(x_n)_n]\in \mathcal{B}$. 
Then it can be easily checked that $\|\cdot\|^{\prime}$ is a well defined C$^*$-norm on 
$\mathcal{B}$. By the uniqueness of the C$^*$-norm, $\| [(x_n)_n]  \|= \| [(x_n)_n]  \|^{\prime}$ 
for any $[(x_n)_n] \in\mathcal{B}$. 
The action $\gamma$ on $A\otimes\mathcal{W}$ induces a natural action 
on $\mathcal{B}$. 
We denote it by the same symbol $\gamma$ for simplicity. 
In this section we shall consider properties of the fixed point algebra $\mathcal{B}^{\gamma}$.

Consider the GNS representation $(\pi_{\tau},H_{\tau})$  of $A\otimes\mathcal{W}$ associated with 
$\tau$. Note that $\pi_{\tau}$ extends to a representation $\overline{\pi}_{\tau}$ 
of $M(A\otimes\mathcal{W})$ on $H_{\tau}$ and 
$\overline{\pi}_{\tau}(M(A\otimes\mathcal{W}))\subset \pi_{\tau}(A\otimes\mathcal{W})''$ 
(see, for example, \cite[3.12]{Ped2}). Put 
$$
M:= \ell^{\infty}(\mathbb{N}, \pi_{\tau}(A\otimes\mathcal{W})'')/\{\{x_n\}_{n\in\mathbb{N}} 
\; |\; \lim_{n\to\omega}\tilde{\tau} (x_n^*x_n)=0\},
$$
and define a homomorphism $\Pi$ from $A$ to $M$ by $\Pi (a):= (\overline{\pi}_{\tau}(a\otimes 
1_{\mathcal{W}^{\sim}}))_n$.  
Note that $M$ is a von Neumann algebraic ultrapower of $\pi_{\tau}(A\otimes\mathcal{W})''$. 
Since $\tau=\tau_A\otimes\tau_\mathcal{W}$, $\pi_{\tau}(A\otimes\mathcal{W})''$ is isomorphic to 
$\pi_{\tau_A}(A)''\bar{\otimes}\pi_{\tau_\mathcal{W}}(\mathcal{W})''$. 
Moreover, $\pi_{\tau}(A\otimes\mathcal{W})''$, $\pi_{\tau_A}(A)''$ and 
$\pi_{\tau_\mathcal{W}}(\mathcal{W})''$ are isomorphic to the AFD II$_1$ factor $\mathcal{R}_0$. 
Set 
$$
\mathcal{M}:= M\cap \Pi (A)'.
$$
It is easy to see that $\mathcal{M}$ is isomorphic to 
$(\mathcal{R}_0\bar{\otimes}\mathcal{R}_0)^{\omega}\cap
(\mathcal{R}_0\bar{\otimes}\mathbb{C})'$ where 
$(\mathcal{R}_0\bar{\otimes}\mathcal{R}_0)^{\omega}$ is the von Neumann algebraic 
ultrapower of $\mathcal{R}_0\bar{\otimes}\mathcal{R}_0$. 

\begin{pro}\label{pro:factor}
With notation as above, $\mathcal{M}$ is a factor of type II$_1$. 
\end{pro}
\begin{proof}
Let $\{N_n\}_{n=1}^\infty$ be an increasing sequence of finite-dimensional subfactors such that 
$\mathcal{R}_{0}=(\bigcup_{n=1}^\infty N_n)''$. 
Since $(\mathcal{R}_0\bar{\otimes}\mathcal{R}_0)\cap (N_n\bar{\otimes}\mathbb{C})'
=(\mathcal{R}_0\cap N_n')\bar{\otimes}\mathcal{R}_0$ is a factor of type II$_1$, the same proof 
as in \cite[Theorem XIV.4.18]{Tak} shows this proposition. 
Indeed, let $(a_n)_n$ be an element in  $\mathcal{M}\setminus \mathbb{C}1_{\mathcal{M}}$.  
By the same way as in \cite[Theorem XIV.4.18]{Tak}, we may assume that 
$\tilde{\tau} (a_n)=0$ for any $n\in\mathbb{N}$ and $\lim_{n\to\infty} \| a_n\|_2>0$ 
where $\| a_n\|_2= \tilde{\tau}(a_n^*a_n)^{1/2}$. 
Using the conditional expectation $\mathcal{E}_n$ from $\mathcal{R}_0\bar{\otimes}\mathcal{R}_0$ 
onto $(\mathcal{R}_0\bar{\otimes}\mathcal{R}_0)\cap (N_n\bar{\otimes}\mathbb{C})'$, 
we can choose $\{k_n\; |\; n\in\mathbb{N}\}\in \omega$ and 
$b_{k_n}\in (\mathcal{R}_0\bar{\otimes}\mathcal{R}_0)\cap (N_n\bar{\otimes}\mathbb{C})'$ 
such that $\tilde{\tau}(b_{k_{n}})=0$ and $\| b_{k_{n}}-a_{k_{n}}\|_2\leq 1/2^{n}$ for any $n\in\mathbb{N}$. 
Since $\tilde{\tau}$ is the unique tracial state on $(\mathcal{R}_0\bar{\otimes}\mathcal{R}_0)\cap 
(N_n\bar{\otimes}\mathbb{C})'$, 
$0=\tilde{\tau}(b_{k_n})\in \overline{\mathrm{co}}
\{ ub_{k_n}u^*\; |\; u\in (\mathcal{R}_0\bar{\otimes}\mathcal{R}_0)\cap (N_n\bar{\otimes}\mathbb{C})',
\; \mathrm{unitary}\}$. 
Therefore there exists a unitary element
 $u_n\in (\mathcal{R}_0\bar{\otimes}\mathcal{R}_0)\cap (N_n\bar{\otimes}\mathbb{C})'$ such that 
$\| b_{k_{n}}- u_n b_{k_n}u_n^*\|_2\geq \| b_{k_n}\|_2/2$ for any $n\in\mathbb{N}$. 
Then we have $(u_n)_n\in \mathcal{M}$ and $\limsup_{n\in\mathbb{N}}\|a_{k_{n}}-u_na_{k_{n}u_n^*}\|_2>0$. 
This shows that $(a_n)_n$ can not be in the center of $\mathcal{M}$.

\end{proof}

The action 
$\tilde{\gamma}=\tilde{\alpha}\otimes \mathrm{id}$ on $\pi_{\tau}(A\otimes\mathcal{W})''\cong
\pi_{\tau_A}(A)''\bar{\otimes}\pi_{\tau_\mathcal{W}}(\mathcal{W})''$ induces an action on $\mathcal{M}$. 
We denote it by the same symbol $\tilde{\gamma}$ for simplicity. 
The following lemma is essentially based on \cite[Proposition 2.1.2]{C2}. 

\begin{lem}\label{lem:central-trivial}
The action $\tilde{\gamma}$ on $\mathcal{M}$ is outer.
\end{lem}
\begin{proof}
It is enough to show that for any element $(u_n)_n$ in 
$(\mathcal{R}_0\bar{\otimes}\mathcal{R}_0)^{\omega}\cap(\mathcal{R}_0\bar{\otimes}\mathbb{C})'$, 
there exists an element $(x_n)_n$ in 
$(\mathcal{R}_0\bar{\otimes}\mathcal{R}_0)^{\omega}\cap(\mathcal{R}_0\bar{\otimes}\mathbb{C})'$ 
such that $(\tilde{\gamma}(x_n))_n\neq (x_n)_n$ and $[(x_n)_n, (u_n)_n]=0$. 

Let $(u_n)_n$ be an element in $(\mathcal{R}_0\bar{\otimes}\mathcal{R}_0)^{\omega}\cap
(\mathcal{R}_0\bar{\otimes}\mathbb{C})'$. By \cite[Theorem XIV.4.16]{Tak}, there exists an 
element $(a_n)_n$ in $\mathcal{R}_0^{\omega}\cap \mathcal{R}_0'$ such that 
$(\tilde{\alpha}(a_n))_n\neq (a_n)_n$ because $\tilde{\alpha}$ is outer and $\mathcal{R}_0$ is 
the AFD II$_1$ factor. 
Put $(x_n)_n:= (a_n\otimes 1_{\mathcal{R}_0})_n$ in 
$(\mathcal{R}_0\bar{\otimes}\mathcal{R}_0)^{\omega}$. 
Then $(\tilde{\gamma}(x_n))_n\neq (x_n)_n$ and $[(x_n)_n ,y]=0$ for any $y\in \mathcal{R}_0
\bar{\otimes}\mathcal{R}_0$. Taking a suitable subsequence of $(x_n)_n$, we obtain the conclusion. 
\end{proof}

By Proposition \ref{pro:factor} and Lemma \ref{lem:central-trivial}, we obtain the following proposition. 

\begin{pro}\label{pro:fixed-factor}
The fixed point algebra $\mathcal{M}^{\tilde{\gamma}}$ is a factor of type II$_1$. 
\end{pro}

Define a homomorphism $\Phi$ from $M(A\otimes\mathcal{W})^{\omega}$ to $M$ by 
$\Phi ((x_n)_n):= (\overline{\pi}_{\tau}(x_n))_n$. 
By Kaplansky's density theorem, we see that $\Phi |_{(A\otimes\mathcal{W})^{\omega}}$ is surjective. 
It is easy to see that $\Phi$ maps $\mathcal{A}$ into  $\mathcal{M}$. 
The following proposition is essentially based on \cite[Theorem 3.3]{KR} and \cite[Theorem 3.1]{MS3}. 

\begin{pro}
The restriction $\Phi|_{\mathcal{A}} : \mathcal{A}\to \mathcal{M}$ is surjective. 
\end{pro}
\begin{proof}
Let $x$ be a contraction in $\mathcal{M}$. Since $\Phi |_{(A\otimes\mathcal{W})^{\omega}}$ 
is surjective, there exists a contraction $(x_n)_n$ in 
$(A\otimes\mathcal{W})^{\omega}$ such that $\Phi ((x_n)_n)=x$. 
Let $D$ be a C$^*$-subalgebra of  $M(A\otimes\mathcal{W})^{\omega}$ generated by 
$(x_n)_n$ and $\{ a\otimes1_{\mathcal{W}^{\sim}} \; |\; a\in A\} $, and 
put $I:=\mathrm{ker}\; \Phi|_{D}$. 
Then the rest of proof is same as the proof of \cite[Theorem 3.1]{MS3}. 
Indeed, let $\{e_{k}\}_{k\in\mathbb{N}}$ be an approximate unit for $I$ which is 
quasicentral for $D$. (Note that $D$ is separable.) 
Since $[(x_n)_n, a\otimes1_{\mathcal{W}^{\sim}}]\in I$, we have 
\begin{align*}
0
& =\lim_{k\to\infty} \| (1-e_{k})[(x_n)_n, a\otimes1_{\mathcal{W}^{\sim}}] (1-e_{k})\| \\
& =\lim_{k\to\infty} \|[(1-e_{k})(x_n)_n(1-e_{k}), a\otimes 1_{\mathcal{W}^{\sim}}]\|
\end{align*}
for any $a\in A$. 
Then we obtain the conclusion by usual arguments 
(see the proof of \cite[Theorem 3.3]{KR} and \cite[Theorem 3.1]{MS3}). 
\end{proof}

Let $\{h_n\}_{n\in\mathbb{N}}$ be an approximate unit for $A$. 
Since $\lim_{n\to \infty} \tau (h_n\otimes 1_{\mathcal{W}})=1$, 
a similar argument as in the proof of \cite[Proposition 2.1]{Na3} shows 
 $\mathcal{I}\subset \mathrm{ker}\; \Phi|_{\mathcal{A}}$. 
Therefore  $\Phi|_{\mathcal{A}}$ induces a surjective homomorphism $\varrho$ from 
$\mathcal{B}$ to $\mathcal{M}$. 
Since $\gamma$ is an action of a finite group, it is easy to show the following proposition. 

\begin{pro}\label{pro:surjective}
The restriction $\varrho|_{\mathcal{B}^{\gamma}}: \mathcal{B}^{\gamma} \to 
\mathcal{M}^{\tilde{\gamma}}$ is surjective. 
\end{pro}

The following lemma is essentially based on \cite[Lemma 3.2]{MS3}. 
This lemma may be considered that a homomorphism $a\mapsto a\otimes 1_{\mathcal{W}^{\sim}}$ 
from $A$ to $M(A\otimes\mathcal{W})^{\omega}$ has ``property (SI) with respect to 
$A\otimes\mathcal{W}$''.

\begin{lem}\label{lem:A-SI}
Let $(x_n)_n$ and $(y_n)_n$ be positive contractions in $\mathcal{A}$ such that  
$$
\lim_{n\to\omega} \tau (x_n)=0 \quad \text{and} \quad \inf_{m\in\mathbb{N}}\lim_{n\to \omega} 
\tau (y_n^m)>0.
$$
Then there exists an element $(s_n)_n$ in $\mathcal{A}$ such that  
$(s_n^*s_n)_n=(x_n)_n$ and $(y_ns_n)_n=(s_n)_n$. 
\end{lem}
\begin{proof}
Similar arguments as in the proofs of \cite[Lemma 3.2]{MS3} and \cite[Theorem 1.1]{MS} with some 
modifications in \cite[Section 5]{Na2} show this lemma. Indeed, let $\varphi$ 
be a pure state on $A$. We can uniquely extend $\varphi$  to a pure state $\tilde{\varphi}$ on 
$A^{\sim}$. Since we may assume that $A$ is a (separable simple) non-type I C$^*$-algebra, 
$K(H_{\tilde{\varphi}})\cap \pi_{\tilde{\varphi}}(A^{\sim})=\{0\}$. 
Therefore \cite[Proposition 5.9]{KR} implies that the identity map on $A^{\sim}$ can be approximated 
in the pointwise norm topology by a completely positive map $\psi$ of the form 
$$
\psi (a)= \sum_{i,j=1}\tilde{\varphi}(d_i^*ad_j)c_i^*c_j, \quad a\in A^{\sim},
$$ 
where $c_i,d_i\in A^{\sim}$. 
Note that $\sum_{i,j=1}\tilde{\varphi}(d_i^*ad_j)(c_i^*c_j\otimes 1_{\mathcal{W}^{\sim}})x_n$ is an element 
in $A\otimes\mathcal{W}$. 
Since $A\otimes\mathcal{W}$ has strict comparison, a similar argument as in the proof of 
\cite[Lemma 3.2]{MS3} (we need to use \cite[Lemma 5.7]{Na2}) shows that there exists a sequence of 
$(s_n)_n$ in $A\otimes\mathcal{W}$ such that $(f_ns_n)_n=(s_n)_n$ and 
$(s_n^*(a\otimes 1_{\mathcal{W}^{\sim}})s_n)_n= ((a\otimes 1_{\mathcal{W}^{\sim}})e_n)$ in 
$(A\otimes\mathcal{W})^{\omega}$ for any $a\in A^{\sim}$. Therefore we obtain the conclusion 
(see \cite[Remark 5.5]{Na2}). 
\end{proof}

For any $[(x_n)_n]\in\mathcal{B}$, let $\tau_{\mathcal{B}}([(x_n)_n]):=\lim_{n\to\omega}\tau (x_n)$. 
By a similar argument as in the proof of \cite[Proposition 2.1]{Na3}, $\tau_{\mathcal{B}}$ is a well 
defined tracial state on $\mathcal{B}$. 
The following proposition is essentially based on \cite[Proposition 4.5]{MS2}. 
See also the proof of \cite[Theorem 4.7]{MS1}.

\begin{pro}\label{pro:target-si}
Let $x$ and $y$ be positive contractions in $\mathcal{B}^{\gamma}$ such that 
$$
\tau_{\mathcal{B}}(x)=0 \quad \text{and} \quad \inf_{m\in\mathbb{N}}\tau_{\mathcal{B}}(y^m)>0.
$$
Then there exists an element $s$ in $\mathcal{B}^{\gamma}$ such that 
$s^*s=x$ and $ys=s$.  
\end{pro}
\begin{proof}
Let $(x_n)_n$ and $(y_n)_n$ be positive contractions in $\mathcal{A}$ such that 
$x=[(x_n)_n]$ and $y=[(y_n)_n]$. Then we have 
$$
(\gamma_g(x_n)-x_n)_n(a\otimes 1_{\mathcal{W}^{\sim}})=0 \quad \text{and} \quad 
(\gamma_g(y_n)-y_n)_n(a\otimes 1_{\mathcal{W}^{\sim}})=0
$$
for any $a\in A$ and $g\in G$. 
Since $\alpha$ is strongly outer, Theorem \ref{thm:weak-rohlin} implies that there exists 
a positive contraction $(f_n)_n$ in $A_{\omega}$ such that 
$$
(\alpha_g(f_n)\alpha_h(f_n))_na=0\quad \text{and} \quad \lim_{n\to\omega}\tau_A(f_n)=\frac{1}{|G|}
$$ 
for any $a\in A$ and $g,h\in G$ with $g\neq h$. Let $\{k_n\}_{n=1}^\infty$ be an approximate unit 
for $\mathcal{W}$. Then we have $(f_n\otimes k_n)_n\in \mathcal{A}$, 
$$
\lim_{n\to\omega}\gamma_g(f_n\otimes k_n)\gamma_h(f_n\otimes k_n)(a\otimes 1_{\mathcal{W}^{\sim}})
=0 \quad \text{and} \quad \lim_{n\to\omega}\tau (f_n\otimes k_n)=\frac{1}{|G|}
$$
for any $a\in A$ and $g,h\in G$ with $g\neq h$.
Using \cite[Lemma 5.6]{Na2} instead of \cite[Lemma 4.6]{MS}, a similar argument as in the proof of 
\cite[Proposition 4.5]{MS2} shows that there exists a positive contraction $(\tilde{y}_n)_n$ in 
$\mathcal{A}$ such that
$$
(\tilde{y}_n)_n\leq (y_n)_n, \quad \inf_{m\in\mathbb{N}}\lim_{n\to \omega} \tau (\tilde{y}_n^m)>0 \quad 
\text{and} \quad \lim_{n\to\omega}\gamma_g(\tilde{y}_n)\gamma_h(\tilde{y}_n)
(a\otimes 1_{\mathcal{W}^{\sim}})=0 
$$
for any $a\in A$ and $g,h\in G$ with $g\neq h$. By Lemma \ref{lem:A-SI}, there exists an element 
$(r_n)_n$ in $\mathcal{A}$ such that $(r_n^*r_n)_n=(x_n)_n$ and $(\tilde{y}_nr_n)_n=(r_n)_n$. 
Since $(y_n)_n$ is a positive contraction and $(\tilde{y}_n)_n\leq (y_n)_n$, we have 
$(y_nr_n)_n=(r_n)_n$. 
Put 
$$
(s_n)_n:= \frac{1}{|G|}\sum_{g\in G} \gamma_g ((r_n)_n)\in\mathcal{A} .
$$
Then we have 
$$
(\gamma_g(s_n)-s_n)_n=0, \quad  
(s_n^*s_n-x_n)_n(a\otimes 1_{\mathcal{W}^{\sim}})=0 \quad
\text{and} \quad 
(y_ns_n-s_n)_n(a\otimes 1_{\mathcal{W}^{\sim}})=0
$$ 
for any $a\in A$ and $g\in G$. Therefore, putting $s:= [(s_n)_n]\in \mathcal{B}^{\gamma}$, we obtain the conclusion. 
\end{proof}

The following proposition is essentially based on \cite[Proposition 4.8]{MS2} and 
\cite[Proposition 3.3]{MS3}. 

\begin{pro}\label{pro:main-section3}
(i) $\tau_{\mathcal{B}}$ is the unique tracial state on $\mathcal{B}^{\gamma}$. \ \\
(ii) $\mathcal{B}^{\gamma}$ has strict comparison. 
\end{pro}
\begin{proof}
(i) By Proposition \ref{pro:fixed-factor} and Proposition \ref{pro:surjective}, it suffices to show that 
if $[(x_n)_n]$ is a positive contraction in $\mathrm{ker}\;\varrho|_{\mathcal{B}^{\gamma}}$, 
then $T([(x_n)_n])=0$ for any tracial state $T$ on $\mathcal{B}^{\gamma}$. 
Note that $[(x_n)_n]^{1/2}\in\mathrm{ker}\;\varrho|_{\mathcal{B}^{\gamma}}$, and hence 
$\tau_{\mathcal{B}}([(x_n)_n])=0$. 
Let $\{e_n\}_{n=1}$ be an approximate unit for $A\otimes\mathcal{W}$. 
Then it is easy to see that for any $m\in\mathbb{N}$, $\tau_{\mathcal{B}}(([(e_n)_n]-[(x_n)_n])^m)=1$. 
By Proposition \ref{pro:target-si}, there exists an element $s_1\in\mathcal{B}^{\gamma}$ such that 
$s_1^*s_1=[(x_n)_n]$ and $([(e_n)_n]-[(x_n)_n])s_1=s_1$. 
Hence we have $s_1s_1^*\leq [(e_n)_n]-[(x_n)_n]$. Since $[(x_n)_n]+s_1s_1^*$ is a positive contraction 
and $\tau_{\mathcal{B}}([(x_n)_n]+s_1s_1^*)=0$, the same argument as above shows that there exists an 
element $s_2\in\mathcal{B}^{\gamma}$ such that $s_2^*s_2=[(x_n)_n]$ and 
$([(e_n)_n]-[(x_n)_n]-s_1s_1^*)s_2=s_2$. Repeating this process, for any $N\in\mathbb{N}$, 
we obtain elements $s_1,s_2,...,s_N$ in $\mathcal{B}^{\gamma}$ such that 
$$
s_i^*s_i=[(x_n)_n] \quad \text{and} \quad [(x_n)_n]+\sum_{i=1}^{N}s_is_i^* \leq [(e_n)_n].
$$
Since $T$ is a tracial state and $[(e_n)_n]$ is a contraction, $(N+1)T([(x_n)_n])\leq 1$. 
Therefore $T([(x_n)_n])=0$. 

(ii) Since $\mathcal{W}\otimes M_n(\mathbb{C})$ is isomorphic to $\mathcal{W}$, it can be easily 
checked that $\mathcal{B}^{\gamma}\otimes M_n(\mathbb{C})$ is isomorphic to 
$\mathcal{B}^{\gamma}$. Hence it is enough to show that if $a$ and $b$ are positive elements 
in $\mathcal{B}^{\gamma}$ with $d_{\tau_{\mathcal{B}}}(a)<d_{\tau_{\mathcal{B}}}(b)$, then there exists 
an element $r$ in $\mathcal{B}^{\gamma}$ such that $r^*br=a$. 
Using Proposition \ref{pro:fixed-factor}, Proposition \ref{pro:surjective} and 
Proposition \ref{pro:target-si} instead of \cite[Lemma 4.2]{MS2}, \cite[Theorem 4.3]{MS2} 
and \cite[Proposition 4.5]{MS2}, the same argument as in the proof of 
\cite[Proposition 4.8]{MS2} shows this. 
We shall recall a sketch of a proof for reader's convenience. 
We may assume that $a$ and $b$ are contractions. 
Let $\tilde{\tau}_{\omega}$ denote the unique tracial state on $\mathcal{M}^{\tilde{\gamma}}$. 
Since $\tau_{\mathcal{B}}=\tilde{\tau}_{\omega}\circ \varrho$, we have 
$$
d_{\tau_{\mathcal{B}}}(a)= \tilde{\tau}_{\omega} (1_{(0, \infty)}(\varrho(a))) \quad \text{and} \quad 
d_{\tau_{\mathcal{B}}}(b)= \tilde{\tau}_{\omega} (1_{(0, \infty)}(\varrho(b)))
$$
where $1_{(0,\infty)}$ is the characteristic function of $(0,\infty)$. Note that 
$1_{(0, \infty)}(\varrho(a))$ and $1_{(0, \infty)}(\varrho(b))$ are projections in 
$\mathcal{M}^{\tilde{\gamma}}$ because $\mathcal{M}^{\tilde{\gamma}}$ is a von Neumann algebra. 
Since $\mathcal{M}^{\tilde{\gamma}}$ is a factor of type II$_1$ and 
$\varrho|_{\mathcal{B}^{\gamma}}$ is surjective, the same argument as in the proof of 
\cite[Proposition 4.8]{MS2} shows that there exist positive contractions $y_1$ and $y_2$ 
in $\mathcal{B}^{\gamma}$ and a projection $p$ in $\mathcal{M}^{\tilde{\gamma}}$ such that 
$$
y_1y_2=0,\quad \varrho(y_1)=p,\quad  \varrho (y_2)=1-p, \quad y_1b=by_1, \quad y_2b=by_2
$$
and there exists $\varepsilon >0$ such that 
$$
\tilde{\tau}_{\omega}(1_{(\varepsilon, \infty)}(\varrho(b)p))>0 \quad \text{and} \quad 
\tilde{\tau}_{\omega}(1_{(0, \infty)}(\varrho(a)))< 
\tilde{\tau}_{\omega}(1_{(\varepsilon, \infty)}(\varrho(b)(1-p))).
$$
Furthermore, there exists a unitary element $v$ in $\mathcal{M}^{\tilde{\gamma}}$ such that 
$$
1_{(0, \infty)}(\varrho(a)) \leq v1_{(\varepsilon, \infty)}(\varrho(b)(1-p))v^*
$$
because $\mathcal{M}^{\tilde{\gamma}}$ is a factor of type II$_1$. 
Since $\varrho|_{\mathcal{B}^{\gamma}}$ is surjective, there exists an element $w$ in 
$\mathcal{B}^{\gamma}$ such that $\varrho (w)=v$. Define continuous functions $g$ and $h$ on $
[0, \infty )$ by $g(t)=\min\{1/\varepsilon, 1/t \}$ and $h(t)=tg(t)$. 
Note that $g(b)$ is an element in $(\mathcal{B}^{\gamma})^{\sim}$ and 
$$
h(t):= \left\{\begin{array}{cl}
t/\varepsilon & \text{if } t\in [0,\varepsilon]   \\
1                 & \text{if } t\in (\varepsilon, \infty ) 
\end{array}
\right..
$$ 
Put $r_1:= y_2^{1/2}g(b)^{1/2}w^*a^{1/2}\in\mathcal{B}^{\gamma}$, then we have 
$r_1^*br_1\leq a$ and $\varrho (r_1^*br_1)=\varrho(a)$. 
Therefore $a-r_1^*br_1$ is a positive contraction in $\mathrm{ker}\; \varrho$, and hence 
we have 
$$
\tau_{\mathcal{B}} (a-r_1^*br_1)=0. 
$$ 
Since 
$
\tau_{\mathcal{B}}((h(b)y_1)^m)= \tilde{\tau}_{\omega} (\varrho(h(b))^mp)
\geq \tilde{\tau}_{\omega}(1_{(\varepsilon, \infty)}(\varrho(b)p)))
$
for any $m\in\mathbb{N}$, we have 
$$
\inf_{m\in\mathbb{N}}\tau_{\mathcal{B}}((h(b)y_1)^m)>0. 
$$
Therefore Proposition \ref{pro:target-si} implies that there exists an element $s$ in 
$\mathcal{B}^{\gamma}$ such that 
$$
s^*s= a-r_1^*br_1 \quad \text{and} \quad h(b)y_1s=s. 
$$
Put $r_2:=y_1^{1/2}g(b)^{1/2}s\in \mathcal{B}^{\gamma}$, then we have $r_1^*br_2=0$ and 
$r_2^*br_2=a-r_1^*br_1$. 
Consequently, put $r=r_1+r_2$, then we have $r^*br=a$. 
\end{proof}

\section{Stable uniqueness theorem}\label{sec:stable-uniqueness}

In this section we shall show a variant of \cite[Corollary 3.8]{Na3} which is based on the 
results in \cite{EN}(see also \cite{EGLN}), \cite{EllK}(see also \cite{G}), \cite{DE1} and \cite{DE2}. 

First, we shall define a homomorphism $\rho$ from $F(A\otimes
\mathcal{W})^{\gamma}$ to $\mathcal{B}^{\gamma}$. 
Let $\{k_n\}_{n=1}^\infty$ be an approximate unit for $\mathcal{W}$ with $k_{n+1}k_{n}=k_{n}$, and 
let $\mathcal{W}_0:=\{k_nbk_n \; |\; n\in\mathbb{N}, b\in\mathcal{W}\}$. Then 
$\mathcal{W}_0$ is a dense self-adjoint subalgebra of $\mathcal{W}$. 
For any $(x_n)_n\in (A\otimes\mathcal{W})_{\omega}$, $a\in A$, $b\in\mathcal{W}$ and 
$N\in\mathbb{N}$, we have 
\begin{align*}
((1_{A^{\sim}}\otimes k_N)x_n(1_{A^{\sim}}\otimes b)(a\otimes 1_{\mathcal{W}^{\sim}}))_n  
& = ((1_{A^{\sim}}\otimes k_Nk_{N+1})x_n(a\otimes b))_n \\
& = ((a\otimes k_{N}k_{N+1}b)x_n)_n \\
& = ((1_{A^{\sim}}\otimes k_N)x_n (a\otimes k_{N+1}b))_n \\
& =  ((a\otimes k_Nk_{N+1})x_n (1_{A^{\sim}}\otimes b))_n \\
& = ((a\otimes 1_{\mathcal{W}^{\sim}})(1_{A^{\sim}}\otimes k_N)x_n(1_{A^{\sim}}\otimes b))_n.
\end{align*}
Hence $((1_{A^{\sim}}\otimes k_N)x_n(1_{A^{\sim}}\otimes b))_n\in \mathcal{A}$. 
For any $[(x_n)_n]\in F(A\otimes\mathcal{W})^{\gamma}$ and $k_Nbk_N\in\mathcal{W}_0$, define 
$$
\rho ([(x_n)_n]\otimes k_{N}bk_{N}):= [((1_{A^{\sim}}\otimes k_N)x_n(1_{A^{\sim}}\otimes bk_N))_n]\in
\mathcal{B}. 
$$
We shall show this is well defined. Let $[(x_n)_n]=[(y_n)_n]\in F(A\otimes\mathcal{W})^{\gamma}$ 
and $k_{N}bk_{N}=k_{N^{\prime}}b^{\prime}k_{N^{\prime}}\in \mathcal{W}_0$. 
For any $a\in A$, we have 
\begin{align*}
& (((1_{A^{\sim}}\otimes k_N)x_n(1_{A^{\sim}}\otimes bk_N)-(1_{A^{\sim}
}\otimes k_{N^{\prime}})y_n(1_{A^{\sim}}\otimes b^{\prime}k_{N^{\prime}}))
(a\otimes 1_{\mathcal{W}^{\sim}}))_n \\
& = ((1_{A^{\sim}}\otimes k_N)x_n(a\otimes bk_N)-(1_{A^{\sim}
}\otimes k_{N^{\prime}})y_n(a\otimes b^{\prime}k_{N^{\prime}}))_n \\
& = ((a\otimes k_{N}bk_{N})x_n-(a\otimes k_{N^{\prime}}b^{\prime}k_{N^{\prime}})y_n)_n=
 ((a\otimes k_{N}bk_{N})(x_n-y_n))_n=0.
\end{align*}
Therefore  
$[((1_{A^{\sim}}\otimes k_N)x_n(1_{A^{\sim}}\otimes bk_N))_n]=
 [((1_{A^{\sim}}\otimes k_{N^{\prime}})y_n(1_{A^{\sim}}\otimes b^{\prime}k_{N^{\prime}}))_n]$. 
By a similar argument, 
it can be easily checked that $\rho$ is a homomorphism from the algebraic tensor product 
$F(A\otimes\mathcal{W})^{\gamma}\odot\mathcal{W}_0$ to $\mathcal{B}^{\gamma}$. 
Since we have 
\begin{align*}
\| \rho ([(x_n)_n]\otimes k_{N}bk_{N})\| 
& = \sup_{a\in A_{+,1}} \lim_{n\to \omega} \|(1_{A^{\sim}}\otimes k_{N})x_n
(1_{A^{\sim}}\otimes bk_{N})(a\otimes 1_{\mathcal{W}^{\sim}}) \| \\ 
& = \sup_{a\in A_{+,1}} \lim_{n\to \omega} \|(a\otimes k_{N}bk_{N})x_n\| \\ 
& \leq \sup_{a\in A_{+,1}} \lim_{n\to \omega} \|a\| \| k_{N}bk_{N}\|  \| x_n\| \\
& = \lim_{n\to\omega} \| x_n\| \cdot \| k_{N}bk_{N}\|,
\end{align*}
$\rho$ can be extended to a homomorphism from the algebraic tensor product 
$F(A\otimes\mathcal{W})^{\gamma}\odot\mathcal{W}$ to $\mathcal{B}^{\gamma}$. 
Consequently, $\rho$ can be extended to a homomorphism from 
$F(A\otimes\mathcal{W})^{\gamma}\otimes\mathcal{W}$ 
to $\mathcal{B}^{\gamma}$ because $\mathcal{W}$ is nuclear. 
By the construction of $\rho$, it is easy to show the following proposition. 

\begin{pro}\label{pro:homrho}
Let $(z_n)_n$ be an element in $\mathcal{A}$ such that 
$[(z_n)_n]=\rho ([(x_n)_n]\otimes b)$ 
for some $[(x_n)_n]\in F(A\otimes\mathcal{W})^{\gamma}$ and 
$b\in \mathcal{W}$. Then 
$$
(z_n(a\otimes 1_{\mathcal{W}^{\sim}}))_n= (x_n(a\otimes b))_n
$$
for any $a\in A$. 
\end{pro}

\begin{rem}
Note that there exists an element $(x_n)_n$ in $(A\otimes\mathcal{W})_{\omega}$ such that  
$(x_n)_n\notin \mathcal{A}$. Indeed, if $a$ is not an element in the center of $A$, 
$(a\otimes (k_n^2-k_n))_n$ is such an element. 
But we do not know whether there exist $(x_n)_n\in (A\otimes \mathcal{W})_{\omega}$ and 
$b\in \mathcal{W}$ such that $(x_n(1_{A^{\sim}}\otimes b))_n\notin \mathcal{A}$.  
\end{rem}

The following lemma is an analogous lemma of \cite[Lemma 3.6]{Na3}. 
\begin{lem}\label{lem:Lemma3.6}
If $x$ is a positive element in $F(A\otimes\mathcal{W})$, then 
$$
\tau_{\mathcal{B}}(\rho (x\otimes b))= \tau_{\omega}(x) \tau_\mathcal{W} (b)
$$
for any $b\in\mathcal{W}$. 
\end{lem}
\begin{proof}
Let $(z_n)_n$ be an element in $\mathcal{A}$ such that $[(z_n)_n]=\rho (x\otimes b)$, and let 
$\{h_n\}_{n=1}^\infty$ be an approximate unit for $A$. 
Note that $\tau_{\mathcal{B}}(\rho (x\otimes b))=\lim_{n\to\omega}\tau (z_n)$. 
Since $\lim_{n\to \infty} \tau (h_n\otimes 1_{\mathcal{W}})=1$, 
a similar argument as in the proof of \cite[Proposition 5.3]{Na2} shows 
$$
\lim_{n\to\omega}\tau (z_n)=\lim_{m\to \infty} \lim_{n\to \omega} 
\tau (z_n(h_m\otimes 1_{\mathcal{W}^{\sim}})).
$$
By Proposition \ref{pro:homrho} and \cite[Lemma 3.6]{Na3}, 
$$
\lim_{n\to \omega} \tau (z_n(h_m\otimes 1_{\mathcal{W}^{\sim}}))= \tau_{\omega}(x)
\tau (h_m\otimes b)=\tau_{\omega}(x)\tau_A (h_m)\tau_\mathcal{W}(b)
$$
for any $m\in\mathbb{N}$. 
Therefore $\tau_{\mathcal{B}}(\rho (x\otimes b))= \tau_{\omega}(x) \tau_\mathcal{W} (b)$ since 
$\lim_{m\to \infty} \tau_A (h_m)=1$. 
\end{proof}

For a projection $p$ in $F(A\otimes\mathcal{W})^{\gamma}$, let
$$
\mathcal{B}^{\gamma}_p:=\overline{\rho (p\otimes s)
\mathcal{B}^{\gamma} \rho (p\otimes s)}
$$
where $s$ is a strictly positive element in $\mathcal{W}$.
Note that $\mathcal{B}^{\gamma}_p$ is a hereditary subalgebra of  $\mathcal{B}^{\gamma}$.
Define a homomorphism $\sigma_p$ from $\mathcal{W}$ to $\mathcal{B}^{\gamma}_p$ by 
$$
\sigma_p (b) = \rho (p\otimes b)
$$
for any $b\in \mathcal{W}$.

Since the target algebra $\mathcal{B}^{\gamma}$ has strict comparison by 
Proposition \ref{pro:main-section3},  
the same proof as \cite[Proposition 3.7]{Na3} shows the following proposition by 
using Lemma \ref{lem:Lemma3.6} instead of \cite[Lemma 3.6]{Na3}. 
See \cite[Definition 3.2]{Na3} for the definition of the $(L,N)$-fullness. 

\begin{pro}\label{pro:full-inclusion}
There exist maps $L: \mathcal{W}_{+,1}\setminus \{0\}\times (0,1)\to \mathbb{N}$ and 
$N: \mathcal{W}_{+,1}\setminus \{0\}\times (0,1) \to (0,\infty)$ such that the following holds. 
If $p$ be a projection in $F(A\otimes\mathcal{W})^{\gamma}$ such that 
$\tau_{\omega} (p)>0$, then $\sigma_p$ is $(L,N)$-full. 
\end{pro}

The following corollary is an immediate consequence of \cite[Proposition 3.3]{Na3} 
and the proposition above.  For finite sets $\mathcal{F}_1$ and $\mathcal{F}_2$, 
let $\mathcal{F}_1\odot \mathcal{F}_2:=\{a\otimes b \; |\; a\in \mathcal{F}_1, b\in \mathcal{F}_2\}$.

\begin{cor}\label{cor:stable-uniqueness} 
Let $\Omega$ be a compact metrizable space. For any finite subsets 
$F_1\subset C(\Omega)$, $F_2\subset \mathcal{W}$ and $\varepsilon>0$, 
there exist finite subsets $\mathcal{F}_1\subset C(\Omega)$, $\mathcal{F}_2\subset \mathcal{W}$, 
$m\in\mathbb{N}$  and $\delta >0$ such that the following holds. 
Let $p$ be a projection in $F(A\otimes\mathcal{W})^{\gamma}$ such that 
$\tau_{\omega} (p)>0$. 
For any contractive ($\mathcal{F}_1\odot \mathcal{F}_2, \delta$)-multiplicative maps 
$\varphi, \psi : C(\Omega)\otimes \mathcal{W}\to \mathcal{B}^{\gamma}_p$, 
there exist a unitary element $u$ in $M_{m^2+1}(\mathcal{B}^{\gamma}_p)^{\sim}$ and 
$z_1,z_2,...,z_m\in\Omega$ such that 
\begin{align*}
\| u & (\varphi(f\otimes b) \oplus  \overbrace{\bigoplus_{k=1}^m f(z_k)\rho (p\otimes b)\oplus \cdots \oplus\bigoplus_{k=1}^m f(z_k)\rho (p\otimes b) }^m) u^* \\
& - \psi(f\otimes b)\oplus \overbrace{\bigoplus_{k=1}^m f(z_k)\rho (p\otimes b) \oplus \cdots \oplus \bigoplus_{k=1}^m f(z_k)\rho (p\otimes b)}^m\| < \varepsilon 
\end{align*}
for any $f\in F_1$ and $b\in F_2$. 
\end{cor}

\section{Classification of normal elements in $F(A\otimes\mathcal{W})^{\gamma}$}
\label{sec:normal}

In this section we shall classify certain normal elements in $F(A\otimes\mathcal{W})^{\gamma}$ 
up to unitary equivalence. Furthermore, we shall consider the comparison theory for certain 
projections in $F(A\otimes\mathcal{W})^{\gamma}$. We assume that $\Omega$ is a compact 
metrizable space in this section. 

Using Proposition \ref{pro:key-pro} and Proposition \ref{pro:MvN-u} instead of 
\cite[Proposition 4.1]{Na3}, \cite[Proposition 4.2]{Na3} and \cite[Proposition 4.8]{Na3}, 
we obtain the following lemma by the same proof as \cite[Lemma 5.1]{Na3}. 
See also \cite[Lemma 4.1]{M2} and \cite[Lemma 4.2]{M2}. 

\begin{lem}\label{lem:m4.2}
Let $F$ be a finite subset of $C(\Omega)$ 
and $\varepsilon >0$. 
Suppose that $\varphi$ and $\psi$ are unital homomorphisms from $C(\Omega)$ to 
$F(A\otimes\mathcal{W})^{\gamma}$ 
such that 
$
\tau_{\omega} \circ \varphi  = \tau_{\omega} \circ \psi .
$ 
Then there exist a projection $p\in F(A\otimes\mathcal{W})^{\gamma}$, 
$(F,\varepsilon)$-multiplicative unital c.p. maps $\varphi^{\prime}$ and $\psi^{\prime}$ from 
$C(\Omega)$ to $pF(A\otimes\mathcal{W})^{\gamma}p$, a unital homomorphism $\sigma$ 
from $C(\Omega)$ to $(1-p)F(A\otimes\mathcal{W})^{\gamma}(1-p)$ with finite-dimensional 
range and a unitary element $u\in F(A\otimes\mathcal{W})^{\gamma}$ 
such that 
$$
0 <\tau_{\omega} (p) < \varepsilon, \;
\| \varphi (f)- (\varphi^{\prime}(f)+ \sigma (f))\| <\varepsilon,  \;
\| \psi (f)- u(\psi^{\prime}(f)+ \sigma (f))u^*\| <\varepsilon 
$$
for any $f\in F$. 
\end{lem}

The following theorem is a variant of \cite[Theorem 5.2]{Na3}. See  also \cite[Theorem 4.5]{M2}. 

\begin{thm}\label{thm:unitary-equivalence-ed}
Let $F_1$ be a finite subset of $C(\Omega)$, 
$F_2$ a finite subset of $A$ and $F_3$ a finite subset of $\mathcal{W}$, and let $\varepsilon >0$. 
Then there exist mutually orthogonal positive elements $h_1,h_2,...,h_{l}$ in $C(\Omega)$ of 
norm one such that the following holds. 
For any $\nu >0$, there exist finite subsets $\mathcal{G}_1\subset C(\Omega)$, 
$\mathcal{G}_2\subset A\otimes \mathcal{W}$ and $\delta >0$ 
such that the following holds. 
If $\varphi$ and $\psi$ are  unital c.p. maps from $C(\Omega)$ to 
$M(A\otimes\mathcal{W})$ such that  
$$
\tau (\varphi (h_i)) \geq \nu, \; \forall i\in\{1,2,...,l\}, 
$$
$$
\| [\varphi (f), x] \| < \delta , \; \| [\psi(f),x ] \| < \delta,  
\; \forall f\in \mathcal{G}_1, x\in \mathcal{G}_2, 
$$
$$
\| (\varphi (f_1f_2)- \varphi (f_1)\varphi (f_2))x\| < \delta, \; 
\| (\psi (f_1f_2)- \psi (f_1)\psi (f_2))x\| < \delta, \; \forall f_1,f_2\in \mathcal{G}_1, 
x\in \mathcal{G}_2,
$$
$$
\| (\gamma_g (\varphi (f))-\varphi (f))x \| < \delta, \; 
\| (\gamma_g (\psi (f))-\psi (f))x \| < \delta, \; \forall g\in G, f\in \mathcal{G}_1, x\in \mathcal{G}_2,
$$
$$
| \tau (\varphi (f)) -\tau (\psi (f)) | < \delta, \; \forall  f\in \mathcal{G}_1,
$$
then there exists a contraction $u$ in $(A\otimes\mathcal{W})^{\sim}$ such that 
$$
\| (a\otimes b) (u^*u -1)\| < \varepsilon, \;
\| (a\otimes b)(uu^* -1) \| < \varepsilon, \;
\| (a\otimes b)(\gamma_g(u)-u) \| < \varepsilon,
$$
$$
\| u\varphi (f)(a\otimes b)u^* - \psi(f)(a\otimes b) \| < \varepsilon 
$$
for any $f\in F_1$, $a\in F_2$, $b\in F_3$ and $g\in G$. 
\end{thm}

\begin{proof}

We may assume that every element in $F_2$ and $F_3$ is positive and of norm one. 
Take positive elements $h_1,h_2,...,h_l$ in $C(\Omega)$ by the same way as  in the proof of 
\cite[Theorem 5.2]{Na3}. We will show that $h_1,h_2,...,h_l$ have the desired property. 
On the contrary, suppose that $h_1,h_2,...,h_l$ did not have the desired property. Then there exists a 
positive number $\nu$ satisfying the following: For any $n\in\mathbb{N}$, there exist unital 
c.p. maps $\varphi_n, \psi_n : C(\Omega)\to M(A\otimes\mathcal{W})$ such that 
$$
\tau (\varphi_n (h_i)) \geq \nu, \: \forall i \in\{1,2,...,l\},
$$
$$
\| [\varphi_n(f_1),x]\| \to 0, \; \| [\psi_n(f_1),x]\| \to 0, \; \| (\varphi_n(f_1f_2)- \varphi_n(f_1)
\varphi_n(f_2))x\| \to 0,  
$$
$$
\| (\psi_n(f_1f_2)- \psi_n (f_1)\psi_n(f_2))x\| \to 0, \; \| (\gamma_g (\varphi_n (f_1))-\varphi_n (f_1))x 
\|\to 0, 
$$
$$
\| (\gamma_g (\psi_n (f_1))-\psi_n (f_1))x \|\to 0, \; |\tau (\varphi_n(f_1))-\tau (\psi_n(f_1))| \to 0
$$
as $n\to\infty$ for any $f_1,f_2\in C(\Omega)$, $x\in A\otimes\mathcal{W}$ and $g\in G$ and 
$$
\max_{f\in F_1, a\in F_2, b\in F_3} 
\| u\varphi_n (f)(a\otimes b)u^* - \psi_n (f)(a\otimes b)\| \geq \varepsilon 
$$
for any contraction $u$ in $(A\otimes\mathcal{W})^{\sim}$ satisfying 
$$
\| (a\otimes b)(\gamma_g(u)-u) \| < \varepsilon, \; 
\|(a\otimes b) (u^*u -1) \| < \varepsilon, \;
\|(a\otimes b) (uu^* -1) \| < \varepsilon
$$
for any $a\in F_2$, $b\in F_3$ and $g\in G$. 

Define homomorphisms $\varphi$ and $\psi$ from $C(\Omega)$ to 
$F(A\otimes\mathcal{W})^{\gamma}$ by 
$\varphi (f) := [(\varphi_n(f))_n]$ and $\psi (f):= [(\psi_n(f))_n]$ for any $f\in C(\Omega)$. 
Then we have 
$$
\tau_{\omega} \circ \varphi= \tau_{\omega} \circ \psi \quad \text{and} \quad 
\tau_{\omega}(\varphi (h_i))\geq \nu 
$$
for any $i=1,2,...,l$. 
 
We obtain finite subsets $\mathcal{F}_1\subset C(\Omega)$, $\mathcal{F}_2\subset \mathcal{W}$, 
$m\in\mathbb{N}$  and $\delta >0$ by applying Corollary \ref{cor:stable-uniqueness} to $F_1$ 
and $F_3$ and $\varepsilon /7$. Put 
$$
F_1^{\prime}:= F_1\cup \mathcal{F}_1 \cup \{h_1, h_2,...,h_l \} \quad \text{and} \quad 
\varepsilon^{\prime}:= 
\min \left\{\frac{\varepsilon}{7}, \frac{\delta}{\max\{\| b\| \; |\; b\in \mathcal{F}_2\}}, 
\frac{\nu}{(m^2+2)} \right\}.
$$

Applying Lemma \ref{lem:m4.2} to $F_1^{\prime}$, $\varepsilon^{\prime}$, $\varphi$ and $\psi$, 
there exist a projection $p\in F(A\otimes\mathcal{W})^{\gamma}$, 
$(F_1^{\prime},\varepsilon^{\prime})$-multiplicative unital c.p. maps $\varphi^{\prime}$ 
and $\psi^{\prime}$ from $C(\Omega)$ to $pF(A\otimes\mathcal{W})^{\gamma}p$, a unital 
homomorphism $\sigma$ from $C(\Omega)$ to $(1-p)F(A\otimes\mathcal{W})^{\gamma}(1-p)$ 
with finite-dimensional range and a unitary element $w\in F(A\otimes\mathcal{W})^{\gamma}$ 
such that 
$$
0 <\tau_{\omega} (p) < \varepsilon^{\prime}, \;
\| \varphi (f)- (\varphi^{\prime}(f)+ \sigma (f))\| <\varepsilon^{\prime},  \;
\| \psi (f)- w(\psi^{\prime}(f)+ \sigma (f))w^*\| <\varepsilon^{\prime} 
$$
for any $f\in F_1^{\prime}$.  
The Choi-Effros lifting theorem implies that there exist sequences of contractive c.p. maps 
$\varphi_n^{\prime}$, $\psi_n^{\prime}$ and $\sigma_n$ from $C(\Omega )$ to 
$A\otimes\mathcal{W}$ such that 
$\varphi^{\prime}(f)=[(\varphi_n^{\prime}(f))_n]$, $\psi^{\prime}(f)=[(\psi_n^{\prime}(f))_n]$ 
and $\sigma (f)=[(\sigma_n (f))_n]$ for any $f\in C(\Omega )$. 
By \cite[Proposition 4.9]{Na3}, there exists a unitary element $(w_n)_n$ in 
$(A\otimes\mathcal{W})^{\sim}_{\omega}$ such that $w=[(w_n)_n]$. Note that we have 
$(x\gamma_g(w_n))_n=(xw_n)_n$ for any $g\in G$ and $x\in A\otimes\mathcal{W}$, 
$$
\lim_{n\to\omega}\| \varphi_n (f)(a\otimes b)- (\varphi_n^{\prime}(f)+ \sigma_n (f))(a\otimes b)\| 
<\frac{\varepsilon}{7}  \eqno{(1)}
$$
and 
$$
\lim_{n\to\omega}\| \psi_n (f)(a\otimes b)- w_n(\psi_n^{\prime}(f)+ \sigma_n (f))(a\otimes b)w_n^*\| 
<\frac{\varepsilon}{7}  \eqno{(2)}
$$
for any $f\in  F_1^{\prime}$, $a\in F_2$ and $b\in F_3$. 

Define c.p. maps $\Phi^{\prime}$ and $\Psi^{\prime}$ from $C(\Omega)\otimes \mathcal{W}$ 
to $\mathcal{B}_p^{\gamma}$ by 
$$
\Phi^{\prime}:= \rho \circ (\varphi^{\prime}\otimes \mathrm{id}_{\mathcal{W}})\quad 
\text{and} \quad 
\Psi^{\prime}:= \rho \circ (\psi^{\prime}\otimes \mathrm{id}_{\mathcal{W}}).
$$
Then $\Phi^{\prime}$ and $\Psi^{\prime}$ are contractive 
$(\mathcal{F}_1\odot \mathcal{F}_2, \delta)$-multiplicative maps. Hence Corollary 
\ref{cor:stable-uniqueness} implies that there exist a unitary element 
$U$ in $M_{m^2+1}(\mathcal{B}^{\gamma}_p)^{\sim}$ and $z_1,z_2,...,z_m\in\Omega$ such that 
\begin{align*}
\| U & (\Phi^{\prime}(f\otimes b) \oplus  \overbrace{\bigoplus_{k=1}^m f(z_k)\rho (p\otimes b)\oplus \cdots \oplus\bigoplus_{k=1}^m f(z_k)\rho (p\otimes b) }^m) U^* \\
& - \Psi^{\prime}(f\otimes b)\oplus \overbrace{\bigoplus_{k=1}^m f(z_k)\rho (p\otimes b) \oplus \cdots \oplus \bigoplus_{k=1}^m f(z_k)\rho (p\otimes b)}^m\| < \frac{\varepsilon}{7} 
\end{align*}
for any $f\in F_1$ and $b\in F_3$. 

Using Proposition \ref{thm:Matui-Sato} instead of \cite[Proposition 4.1]{Na3}, the same argument 
as in the proof of \cite[Theorem 5.2]{Na3} shows that there exist mutually orthogonal projections 
$\{p_{j,k} \}_{j,k=1}^m$ in $(1-p)F(A\otimes\mathcal{W})^{\gamma}(1-p)$ and a homomorphism 
$\sigma^{\prime\prime}: C(\Omega)\to (1-p-q)F(A\otimes\mathcal{W})^{\gamma}(1-p-q)$ where 
$q=\sum_{j,k=1}^mp_{j,k}$ such that 
$$
\| \sigma (f) - \left(\sum_{j=1}^m\sum_{k=1}^mf(z_k)p_{j,k} + \sigma^{\prime\prime} (f) \right) \| 
< \frac{2\varepsilon}{7}
$$
for any $f\in F_1$ and $p_{j,k}$ is Murray-von Neumann equivalent to $p$ for any $j,k=1,2,...,m$. 
Define a homomorphism $\hat{\sigma}$ from $C(\Omega)$ to 
$F(A\otimes\mathcal{W})^{\gamma}$ by 
$$
\hat{\sigma}(f):= \sum_{j=1}^m\sum_{k=1}^mf(z_k)p_{j,k}+ \sigma^{\prime\prime} (f)
$$ 
for any $f\in C(\Omega)$. 
By the Choi-Effros lifting theorem, there exists a sequence of contractive 
c.p. maps $\hat{\sigma}_n$ from $C(\Omega)$ to $A\otimes\mathcal{W}$ such that 
$\hat{\sigma}(f)=[(\hat{\sigma}_n (f))_n]$. Note that we have 
$$
\lim_{n\to \omega} \|\sigma_n(f)(a\otimes b)- \hat{\sigma}_n(f)(a\otimes b) \|< \frac{2\varepsilon}{7} 
 \eqno{(3)}
$$
for any $f\in F_1$, $a\in F_2$ and $b\in F_3$. 

Since we can regard 
$\Phi^{\prime}(f\otimes b)+ \sum_{j=1}^m\sum_{k=1}^mf(z_k)\rho (p_{j,k}\otimes b)\in 
\mathcal{B}_{p+q}^{\gamma}$ as an element in  $M_{m^2+1}(\mathcal{B}^{\gamma}_p)$, 
the same argument as in the proof of \cite[Theorem 5.2]{Na3} shows that there exists a unitary 
element $V$ in $(\mathcal{B}^{\gamma})^{\sim}$ such that 
$$
\| V(\Phi^{\prime} (f\otimes b) + \rho (\hat{\sigma}(f)\otimes b))V^*- (\Psi^{\prime}(f\otimes b)
+\rho (\hat{\sigma}(f)\otimes b)) \|< \frac{\varepsilon}{7}
$$
for any $f\in F_1$ and $b\in F_3$. 
Let $(v_n)_n$ be a contraction in $\mathcal{A}^{\sim}$ such that $V=[(v_n)_n]$. 
Then we have $((a\otimes 1_{\mathcal{W}^{\sim}})v_n^*v_n)_n
=((a\otimes 1_{\mathcal{W}^{\sim}})v_nv_n^*)_n=a\otimes 1_{\mathcal{W}^{\sim}}$ and 
$((a\otimes 1_{\mathcal{W}^{\sim}})\gamma_g(v_n))_n=((a\otimes 1_{\mathcal{W}^{\sim}})v_n)_n$ 
for any $g\in G$ and $a\in A$. 
Furthermore, we see that 
$$
\lim_{n\to\omega} \|v_n(\varphi^{\prime}_n(f)+\hat{\sigma}_n(f))(a\otimes b) v_n^*  
-(\psi^{\prime}_n(f)+\hat{\sigma}_n(f))(a\otimes b)\| < \frac{\varepsilon}{7}  \eqno{(4)}
$$
for any $f\in F_1$, $a\in F_2$ and $b\in F_3$ by Proposition \ref{pro:homrho}. 

Put $(u_n)_n:=(w_nv_n)_n\in ((A\otimes\mathcal{W})^{\sim})^{\omega}$.  Then we have  
$$
((a\otimes b)u_n^*u_n)_n= ((a\otimes b)v_n^*v_n)_n= a\otimes b
$$ 
and 
$$
((a\otimes b)u_nu_n^*)_n=(w_n(a\otimes b)v_nv_n^*w_n)_n=(w_n(a\otimes b)w_n^*)_n=a\otimes b
$$
for any $a\in A$ and $b\in\mathcal{W}$.  
Also, we have 
\begin{align*}
((a\otimes b)\gamma_g(u_n))_n
& =((a\otimes b)\gamma_g(w_n)\gamma_g(v_n))_n=((a\otimes b)w_n\gamma_g(v_n))_n \\
& =(w_n(a\otimes b)\gamma_g(v_n))_n=(w_n(a\otimes b)v_n)_n=((a\otimes b)u_n)_n
\end{align*}
for any $g\in G$, $a\in A$ and $b\in\mathcal{W}$. By (1), (2), (3) and (4), we see that 
$$
\lim_{n\to\omega} \| u_n\varphi_n(f)(a\otimes b)u_n^* - \psi_n (f)(a\otimes b) \| < \varepsilon 
$$
for any $f\in F_1$, $a\in F_2$ and $b\in F_3$. Therefore, taking a sufficiently large $n$, we obtain 
a contradiction. 
Consequently, the proof is complete. 
\end{proof}

The following theorem is the main result in this section. 

\begin{thm}\label{thm:classification-normal}
Let $N_1$ and $N_2$ be normal elements in $F(A\otimes\mathcal{W})^{\gamma}$ such that 
$\mathrm{Sp} (N_1)=\mathrm{Sp} (N_2)$ and 
$\tau_{\omega} (f(N_1)) >0$ for any $f\in C(\mathrm{Sp}(N_1))_{+}\setminus \{0\}$. 
Then there exists a unitary element $u$ in $F(A\otimes\mathcal{W})^{\gamma}$ 
such that $uN_1u^* =N_2$ if and only if 
$
\tau_{\omega} (f(N_1))= \tau_{\omega} (f(N_2))
$ 
for any $f\in C(\mathrm{Sp}(N_1))$. 
\end{thm}
\begin{proof}
By a similar argument as in the proof of \cite[Theorem 5.3]{Na3}, 
we can prove this theorem. We shall give a proof for reader's convenience. 
 
Since the only if part is clear, we will show the if part. 
Define unital homomorphisms $\varphi$ and $\psi$ from $C(\mathrm{Sp}(N_1))$ to 
$F(A\otimes\mathcal{W})^{\gamma}$ by $\varphi (f):= f(N_1)$ and $\psi (f):=f(N_2)$, respectively.  
By the Choi-Effros lifting theorem, we see that there exist sequences of unital c.p. maps 
$\varphi_n$ and $\psi_n$ from $C(\mathrm{Sp}(N_1))$ to $(A\otimes\mathcal{W})^{\sim}$ such that 
$f(N_1)=[(\varphi_n(f))_n]$ and $f(N_2)=[(\psi_n(f))_n]$ for any $f\in C(\mathrm{Sp}(N_1))$. 
Then we have 
$$
|\tau (\varphi_n(f_1))-\tau_{\omega}(f_1(N_1))| \to 0, \; \| [\varphi_n(f_1),x]\| \to 0, \; 
\| [\psi_n(f_1),x]\| \to 0, 
$$
$$
\| (\varphi_n(f_1f_2)- \varphi_n(f_1)\varphi_n(f_2))x\| \to 0,\;   
\| (\psi_n(f_1f_2)- \psi_n (f_1)\psi_n(f_2))x\| \to 0, 
$$
$$
\| (\gamma_g (\varphi_n (f_1))-\varphi_n (f_1))x \|\to 0,  \; 
\| (\gamma_g (\psi_n (f_1))-\psi_n (f_1))x \|\to 0, 
$$
$$
\; |\tau (\varphi_n(f_1))-\tau (\psi_n(f_1))| \to 0
$$
as $n\to\omega$ for any $f_1,f_2\in C(\mathrm{Sp}(N_1))$, $x\in A\otimes\mathcal{W}$ and $g\in G$. 

We denote by $\iota$ the identity function on $\mathrm{Sp}(N_1)$, that is, $\iota (z)=z$ for any 
$z\in\mathrm{Sp}(N_1)$. 
Let $F_1:=\{1, \iota \}\subset C(\mathrm{Sp}(N_1))$, and let 
$\{F_{2,k}\}_{k\in\mathbb{N}}$ and $\{F_{3,k}\}_{k\in\mathbb{N}}$ be increasing sequences 
of finite subsets in $A$ and $\mathcal{W}$ such that $A=\overline{\bigcup_{k\in\mathbb{N}} F_{2,k}}$ 
and $\mathcal{W}=\overline{\bigcup_{k\in\mathbb{N}} F_{3,k}}$, respectively. 
For any $k\in\mathbb{N}$, we obtain mutually orthogonal positive elements 
$h_{1,k},h_{2,k},...,h_{l(k),k}$ in $C(\mathrm{Sp}(N_1))$ of norm one by applying 
Theorem \ref{thm:unitary-equivalence-ed} to $F_1$, $F_{2,k}$, $F_{3,k}$ and $1/k$. 
Put 
$$
\nu_k := \frac{1}{2} \min\{\tau_{\omega} (h_{1,k}(N_1)),\tau_{\omega} (h_{2,k}(N_1)),...,
\tau_{\omega} (h_{l(k),k}(N_1)) \} >0. 
$$ 
Applying Theorem \ref{thm:unitary-equivalence-ed} to $\nu_k$, 
we obtain finite subsets $\mathcal{G}_{1,k}\subset C(\mathrm{Sp}(N_1))$, $\mathcal{G}_{2,k}
\subset A\otimes\mathcal{W}$ and $\delta_k>0$. We may assume that 
$\{\mathcal{G}_{1,k}\}_{k\in\mathbb{N}}$ and $\{\mathcal{G}_{2,k}\}_{k\in\mathbb{N}}$
are increasing sequences and $\delta_k>\delta_{k+1}$ for any $k\in\mathbb{N}$. 
We can find a sequence $\{X_k\}_{k=1}^\infty$ of elements in $\omega$ such that 
$X_k\subset X_{k+1}$ and for any $n\in X_{k}$, 
$$
|\tau (\varphi_n(h_{i,k})) - \tau_{\omega}(h_{i,k}(N_1)) | < \nu_k , \;
\| [\varphi_n(f_1),x]\| < \delta_k, \; \| [\psi_n(f_1),x]\| < \delta_k, 
$$ 
$$
\| (\varphi_n (f_1f_2)- \varphi_n (f_1)\varphi_n (f_2))x\| < \delta_k, \; 
\| (\psi_n (f_1f_2)- \psi_n (f_1)\psi_n (f_2))x\| < \delta_k, 
$$
$$
\| (\gamma_g (\varphi_n (f_1))-\varphi_n (f_1))x \| < \delta_k, \; 
\| (\gamma_g (\psi_n (f_1))-\psi_n (f_1))x \| < \delta_k, 
$$
$$
| \tau (\varphi_n (f_1)) -\tau (\psi_n (f_1)) | < \delta_k
$$
for any $i\in\{1,2,...,l(k)\}$, $f_1,f_2\in\mathcal{G}_{1,k}$, $x\in\mathcal{G}_{2,k}$ and $g\in G$. 
Since we have 
$$
\tau (\varphi_n(h_{i,k})) > \tau_{\omega}(h_{i,k}(N_1))-\nu_k  \geq 2\nu_k - \nu_k = \nu_k 
$$ 
for any $i\in\{1,2,...,l(k)\}$,  
Theorem \ref{thm:unitary-equivalence-ed} implies that for any $n\in X_{k}$, there exists a 
contraction $u_{k,n}$ in 
$(A\otimes\mathcal{W})^{\sim}$ such that 
$$
\| (a\otimes b) (u_{k,n}^*u_{k,n}-1)\| < \frac{1}{k}, \; \|  (a\otimes b)(u_{k,n}u_{k,n}^*-1)\| < \frac{1}{k}, 
$$
$$ 
\|(a\otimes b)(\gamma_g(u_{k,n})-u_{k,n})\| <\frac{1}{k}, \; 
\| u_{k,n}\varphi_n (f)(a\otimes b)u_{k,n}^* - \psi_n(f)(a\otimes b) \| <\frac{1}{k}
$$
for any $f\in F_1$, $a\in F_{2,k}$, $b\in F_{3,k}$ and $g\in G$. Since $F_1=\{1, \iota \}$, we have 
$$
\| [u_{k,n}, a\otimes b]\| \leq \| u_{k,n} (a\otimes b)(1-u_{k,n}^*u_{k,n})\| 
+ \| (u_{k,n} (a\otimes b)u_{k,n}^*- a\otimes b)u_{k,n} \|< \frac{2}{k} 
$$
and 
$$
\| u_{k,n}\varphi_n (\iota) (a\otimes b)u_{k,n}^*- \psi_n(\iota)(a\otimes b)\| < \frac{1}{k}
$$
for any $n\in X_{k}$, $a\in F_{2,k}$ and $b\in F_{3,k}$. 
Put 
$$
u_{n} := \left\{\begin{array}{cl}
1 & \text{if } n\notin X_1   \\
u_{k,n} & \text{if } n\in X_k\setminus X_{k+1}\quad (k\in\mathbb{N})
\end{array}
\right..
$$
Then 
$$
\| (a\otimes b)(u_n^*u_n-1)\|\to 0, \; 
\| (a\otimes b)(u_nu_n^*-1)\|\to 0, \; \| (a\otimes b)(\gamma_g (u_n)-u_n)\|\to 0,
$$
$$
\| [u_n, a\otimes b]\|\to 0,\; \|(u_n\varphi_n(\iota)u_n^*-\psi_n(\iota))
(a\otimes b)\|\to 0
$$
as $n\to \omega$ for any $a\in A$, $b\in\mathcal{W}$ and $g\in G$. Therefore $[(u_n)_n]$ is 
a unitary element in $F(A\otimes\mathcal{W})^{\gamma}$ and $[(u_n)_u]N_1[(u_n)_n]^*=N_2$. 
\end{proof}

Applying the theorem above to projections, we obtain the following corollary. 
Note that if $p$ is a projection, then $C(\mathrm{Sp}(p))$ can be identified with 
$\{\lambda_1 p+ \lambda_2(1-p)\; |\; \lambda_1,\lambda_2\in\mathbb{C}\}$. 
Hence it is clear that $\tau_{\omega}(f(p))>0$ for any $f\in C(\mathrm{Sp}(p))_{+}\setminus \{0\}$ 
if and only if $0<\tau_{\omega}(p)<1$. Also, for projections $p$ and $q$, 
we have  $\tau_{\omega}(f(p))=\tau_{\omega}(f(q))$ for any $f\in C(\mathrm{Sp}(p))$ 
if and only if $\tau_{\omega}(p)=\tau_{\omega}(q)$.

\begin{cor}\label{thm:comparison}
Let $p$ and $q$ be projections in $F(A\otimes \mathcal{W})^{\gamma}$ such that 
$0< \tau_{\omega} (p) <1$.  Then $p$ and $q$ are unitarily equivalent if and only if 
$
\tau_{\omega} (p)= \tau_{\omega} (q)
$. 
\end{cor}

The following corollary is important in the next section. 

\begin{cor}\label{cor:comparison}
Let $p$ and $q$ be projections in $F(A\otimes \mathcal{W})^{\gamma}$ such that 
$0< \tau_{\omega} (p) \leq 1$.  Then $p$ and $q$ are Murray-von Neumann equivalent if and only if 
$
\tau_{\omega} (p)= \tau_{\omega} (q)
$. 
\end{cor}
\begin{proof}
By Corollary \ref{thm:comparison}, it suffices to show that if $p$ is a projection in 
$F(A\otimes \mathcal{W})^{\gamma}$ such that $\tau_{\omega}(p)=1$, then $p$ is 
Murray-von Neumann equivalent to $1$. 
Proposition \ref{pro:key-pro} implies that there exists a projection $r$ in 
$F(A\otimes\mathcal{W})^{\gamma}$ such that $r\leq p$ and $\tau_{\omega}(r)=1/2$. 
By Corollary \ref{thm:comparison}, $p-r$ is unitarily equivalent to $1-r$. Therefore 
$p=(p-r)+r$ is Murray-von Neumann equivalent to $(1-r)+r=1$. 
\end{proof}

\section{Rohlin type theorem}\label{sec:main}
In this section we shall show that $\gamma$ has the Rohlin property.

For a $\gamma$-cocycle $w$ in $F(A\otimes\mathcal{W})$, 
define an action $\gamma^{w}$ on $F(A\otimes\mathcal{W})\otimes M_{2}(\mathbb{C})$ by 
$$
\gamma^{w}_g := \mathrm{Ad}\left(\left(\begin{array}{cc}
               1   &   0    \\ 
               0   &   w(g)    
 \end{array} \right)\right) \circ (\gamma_g\otimes \mathrm{id})
$$
for any $g\in G$. 
Since $\gamma$ has the weak Rohlin property, we obtain the following lemma by 
similar arguments as in  \cite[Proposition 4.8]{MS2} and \cite[Proposition 3.3]{MS3} 
(see also arguments in Section \ref{sec:target}). We leave the proof to the reader. 

\begin{lem}\label{lem:fixed-strict-comparison}
Let $a$ and $b$ be positive elements in 
$(F(A\otimes\mathcal{W})\otimes M_{2}(\mathbb{C}))^{\gamma^{w}}$ such that 
$d_{\tau_{\omega}\otimes\mathrm{Tr}_2}(a) <d_{\tau_{\omega}\otimes\mathrm{Tr}_2}(b)$ where 
$\mathrm{Tr}_2$ is the (unnormalized) usual trace on $M_2(\mathbb{C})$. 
Then there exists an element $r$ in 
$(F(A\otimes\mathcal{W})\otimes M_{2}(\mathbb{C}))^{\gamma^{w}}$ such that $r^*br=a$. 
\end{lem}

The proof of the following lemma is based on Connes' $2\times 2$ 
matrix trick in \cite[Corollary 2.6]{C3}.

\begin{lem}\label{lem:cohomology}
Every $\gamma$-cocycle $w$ in $F(A\otimes\mathcal{W})$ is a coboundary. 
\end{lem}
\begin{proof}
Let $\varepsilon >0$. By Proposition \ref{pro:key-pro}, there exists a projection $p_{\varepsilon}$ 
in $F(A\otimes\mathcal{W})^{\gamma}$ such that $\tau_{\omega}(p_{\varepsilon})=1-\varepsilon$. 
Taking a suitable subsequence of a representative of $p_{\varepsilon}$, we may assume that 
$w(g)p_{\varepsilon}=p_{\varepsilon}w(g)$ for any $g\in G$. 
Lemma \ref{lem:fixed-strict-comparison} implies that there exists an element $R_{\varepsilon}$ in 
$(F(A\otimes\mathcal{W})\otimes M_{2}(\mathbb{C}))^{\gamma^{w}}$ such that 
$$
R_{\varepsilon}^*\left(\begin{array}{cc}
               1   &   0    \\ 
               0   &   0    
 \end{array} \right)
R_{\varepsilon} = \left(\begin{array}{cc}
               0   &   0    \\ 
               0   &   p_{\varepsilon}    
 \end{array} \right) .
$$
The diagonal argument shows that there exist a projection $p$ in $F(A\otimes\mathcal{W})^{\gamma}$ 
and an element $R$ in $(F(A\otimes\mathcal{W})\otimes M_{2}(\mathbb{C}))^{\gamma^{w}}$ 
such that $\tau_{\omega}(p)=1$ and 
$$
R^*\left(\begin{array}{cc}
               1   &   0    \\ 
               0   &   0    
 \end{array} \right)
R = \left(\begin{array}{cc}
               0   &   0    \\ 
               0   &   p 
 \end{array} \right) .
$$
By Corollary \ref{cor:comparison}, there exists an element $s$ in 
$F(A\otimes\mathcal{W})^{\gamma}$ such that $s^*s=1$ and $ss^*=p$. 
Taking suitable subsequences of representatives of $s$, $p$ and $R$, we may assume that 
$w(g)s=sw(g)$ for any $g\in G$ and 
$$
\left(\begin{array}{cc}
               0   &   0    \\ 
               0   &   s^*    
 \end{array} \right)
R^*\left(\begin{array}{cc}
               1   &   0    \\ 
               0   &   0    
 \end{array} \right)
R 
\left(\begin{array}{cc}
               0   &   0    \\ 
               0   &   s    
 \end{array} \right)
= \left(\begin{array}{cc}
               0   &   0    \\ 
               0   &   1 
 \end{array} \right) .
$$
It it easy to see that there exists a projection $q$ in $F(A\otimes\mathcal{W})^{\gamma}$ such that 
$\tau_{\omega}(q)=1$ and 
$$
\left(\begin{array}{cc}
               q   &   0    \\ 
               0   &   0    
 \end{array} \right)
=
\left(\begin{array}{cc}
               1   &   0    \\ 
               0   &   0    
 \end{array} \right)
R
\left(\begin{array}{cc}
               0   &   0    \\ 
               0   &   p    
 \end{array} \right)
R^*
\left(\begin{array}{cc}
               1   &   0    \\ 
               0   &   0    
 \end{array} \right) . 
$$
By Corollary \ref{cor:comparison}, there exists an element $t$ in 
$F(A\otimes\mathcal{W})^{\gamma}$ such that $t^*t=1$ and $tt^*=q$. 
Put 
$$
V:= \left(\begin{array}{cc}
               0    &   0    \\ 
               0    &   s^*    
 \end{array} \right)
R^* 
\left(\begin{array}{cc}
               t   &   0    \\ 
               0   &   0    
 \end{array} \right) .
$$
Then we have $V\in (F(A\otimes\mathcal{W})\otimes M_{2}(\mathbb{C}))^{\gamma^{w}}$, 
$$
V^*V=\left(\begin{array}{cc}
               1   &   0    \\ 
               0   &   0    
 \end{array} \right)
\quad \text{and} \quad 
VV^*=\left(\begin{array}{cc}
               0   &   0    \\ 
               0   &   1    
 \end{array} \right) .
$$
It is easy to see that there exists a unitary element $v$ in $F(A\otimes\mathcal{W})$ such that 
$$
V= \left(\begin{array}{cc}
               0   &   0    \\ 
               v   &   0    
 \end{array} \right) .
$$
Since $V\in (F(A\otimes\mathcal{W})\otimes M_{2}(\mathbb{C}))^{\gamma^{w}}$, 
$w(g)\gamma_g(v)=v$ for any $g\in G$. Consequently, $w$ is a coboundary. 
\end{proof}

\begin{rem}
The lemma above shows that the first cohomology of $\gamma$ vanishes. 
This property is one of the important properties for the Bratteli-Elliott-Evans-Kishimoto 
intertwining argument (see, for example, \cite{EK} and \cite{Kis1}) in the classification of 
Rohlin actions. 
\end{rem}

The following theorem is the main result in this paper.

\begin{thm}\label{thm:main} 
Let $A$ be a simple separable nuclear 
C$^*$-algebra with a unique tracial state and no unbounded traces, and 
let $\alpha$ be a strongly outer action of a finite group $G$ on $A$. Then 
$\gamma=\alpha\otimes\mathrm{id}$ on $A\otimes\mathcal{W}$ has 
the Rohlin property. 
\end{thm}
\begin{proof} 
We identify $B(\ell^2(G))$ with $M_{|G|}(\mathbb{C})$. Also, 
we can identify $F(A\otimes\mathcal{W})^{\gamma}$ with 
$F(A\otimes \mathcal{W}\otimes 
\bigotimes_{n\in\mathbb{N}} M_{|G|}(\mathbb{C}))^{\gamma\otimes\mathrm{id}}$ 
because $\mathcal{W}$ is UHF stable. 
Let $\lambda$ be the left regular representation of $G$ on $\ell^2(G)$. 
Define a map $w$  from $G$ to 
$F(A\otimes\mathcal{W})^{\gamma}$ by 
$$
w (g) := [(h_n\otimes k_n \otimes 
\overbrace{1\otimes \cdots \otimes 1}^n \otimes \lambda(g) \otimes 1\otimes \cdots )_n ]
$$
where $\{h_n\}_{n=1}^\infty$ and $\{k_n\}_{n=1}^\infty$ are approximate units for $A$ and 
$\mathcal{W}$, respectively. 
Then $w$ is a homomorphism, and hence $w$ is a $\gamma$-cocycle in $F(A\otimes\mathcal{W})$. 
By Lemma \ref{lem:cohomology}, there exists a unitary element $v$ in 
$F(A\otimes\mathcal{W})$ such that $w(g)=v\gamma_g(v^*)$ for any $g\in G$. 
For any $g\in G$,  let $e_{g}$ be a projection onto $\mathbb{C}\delta_g$ where 
$\{\delta_h\; |\; h\in G\}$ is the canonical basis of $\ell^2(G)$, and put 
$$
p_g:=  v^*[(h_n\otimes k_n \otimes 
\overbrace{1\otimes \cdots \otimes 1}^n \otimes e_g \otimes 1\otimes \cdots )_n ]v.
$$
Then $\{p_{g}\}_{g\in G}$ is a partition of unity in $F(A\otimes\mathcal{W})$ consisting 
of projections satisfying 
$$
\gamma_{g} (p_{h}) =p_{gh}
$$
for any $g,h\in G$. Consequently, $\gamma$ has the Rohlin property. 
\end{proof}

Combining the theorem above and the classification results in \cite{CE} and \cite{EGLN}, 
we obtain the following corollary. 

\begin{cor}\label{cor:main}
Let $A$ and $B$ be simple separable nuclear C$^*$-algebras with a unique tracial state and 
no unbounded traces, and let $\alpha$ and $\beta$ be strongly outer actions of a finite group 
$G$ on $A$ and $B$, respectively. 
Then $\alpha\otimes\mathrm{id}$ on $A\otimes\mathcal{W}$ is conjugate to 
$\beta\otimes\mathrm{id}$ on $B\otimes\mathcal{W}$. 
\end{cor}
\begin{proof}
By \cite[Theorem 6.1]{CE}, $A\otimes\mathcal{Z}$ and $B\otimes\mathcal{Z}$ have finite nuclear 
dimension. Hence \cite[Corollary 6.7]{EGLN} implies that $A\otimes\mathcal{W}$ and 
$B\otimes\mathcal{W}$ are isomorphic to $\mathcal{W}$.  
Therefore we obtain the conclusion by Theorem \ref{thm:classification} and Theorem \ref{thm:main}. 
\end{proof}

\begin{rem}
(1) If $\alpha^{\prime}$ is not a strongly outer action of a non-trivial finite group $G$ on $A$, then 
$(A\otimes\mathcal{W})\rtimes_{\alpha^{\prime}\otimes\mathrm{id}} G$ has at least two extremal 
tracial state. Hence $\alpha^{\prime}\otimes\mathrm{id}$ is not (cocycle) conjugate to the action 
in the corollary above. 
\ \\
(2) There exist uncountably many non-conjugate strongly outer actions of 
$\mathbb{Z}_2$ on $\mathcal{W}$ by \cite[Example 5.6]{Na4} and \cite[Remark 5.7]{Na4}. 
\ \\
(3) For generalizing the corollary above to amenable group actions, it seems to be important that 
we characterize $\mathcal{W}$ by using the central sequence C$^*$-algebra $F(\mathcal{W})$. 
\ \\
(4) If $\alpha$ is a strongly outer action of a finite group $G$ on a simple separable nuclear 
C$^*$-algebra $A$ with a unique tracial state and no unbounded traces, then 
$A\rtimes_{\alpha} G$ is a simple separable nuclear C$^*$-algebra with a unique tracial state and 
no unbounded traces. Hence $(A\otimes\mathcal{W})\rtimes_{\alpha\otimes\mathrm{id}}G
\cong (A\rtimes_{\alpha} G)\otimes\mathcal{W}$ is isomorphic to $\mathcal{W}$. \ \\
(5) We need the unique trace property of $A$ for $A\otimes\mathcal{W}\cong \mathcal{W}$. 
Moreover, we used this property for many arguments in Section \ref{sec:target}. 
But it seems to be possible that we generalize some results in this paper to more general 
$KK$-contractible C$^*$-algebras by suitable modifications. 
\end{rem}

\end{document}